\documentclass[11pt]{amsart}

\usepackage{amsmath,amsthm,amssymb,amssymb, paralist, xspace, graphicx, url, amscd, euscript, mathrsfs,stmaryrd,epic,eepic,color, longtable}

\usepackage[all]{xy}
\usepackage{amsmath,amscd}
\SelectTips{cm}{}
\usepackage{epsfig}
\usepackage{here}
\usepackage{graphicx}
\usepackage[all]{xy}
\usepackage{psfrag}
\usepackage{graphicx,transparent}
\usepackage{enumerate}
\usepackage{caption}

\usepackage[colorlinks=true]{hyperref}

\usepackage{pstricks}

\numberwithin{equation}{subsection}

1

\numberwithin{subsection}{section}

\allowdisplaybreaks[1]

\newtheorem*{namedtheorem}{\theoremname}
\newcommand{\theoremname}{testing}

\theoremstyle{plain}

\newtheorem{thm}{Theorem}[section]
\newtheorem{proposition}[thm]{Proposition}
\newtheorem{proposition-definition}[thm]{Proposition-Definition}
\newtheorem{lemma-definition}[thm]{Lemma-Definition}
\newtheorem{corollary}[thm]{Corollary}
\newtheorem{lemma}[thm]{Lemma}

\theoremstyle{definition}
\newtheorem{definition}[thm]{Definition}

\newtheorem{assumption}[thm]{Assumption}

\newtheorem{remark}[thm]{Remark}

\theoremstyle{remark}

\numberwithin{thm}{section}

\newcommand\ocM{\overline{\mathcal{M}}}

\newcommand\ocN{\overline{\mathcal{N}}}

\newcommand\tS{\tilde{S}}

\newcommand\tC{\tilde{C}}

\newcommand\cA{\mathcal{A}}

\newcommand\cC{\mathcal{C}}

\newcommand\cG{\mathcal{G}}
\newcommand\cH{\mathcal{H}}

\newcommand\cM{\mathcal{M}}
\newcommand\cN{\mathcal{N}}
\newcommand\cO{\mathcal{O}}

\newcommand\cQ{\mathcal{Q}}
\newcommand\cR{\mathcal{R}}

\newcommand\cT{\mathcal{T}}
\newcommand\cU{\mathcal{U}}

\newcommand\uC{\underline{C}}

\newcommand\uG{\underline{G}}

\newcommand\uL{\underline{L}}
\newcommand\uM{\underline{M}}

\newcommand\uQ{\underline{Q}}
\newcommand\uR{\underline{R}}
\newcommand\uS{\underline{S}}
\newcommand\uT{\underline{T}}
\newcommand\uU{\underline{U}}

\newcommand\uX{\underline{X}}

\newcommand\uZ{\underline{Z}}

\newcommand\CC{\mathbb{C}}

\newcommand\NN{\mathbb{N}}

\newcommand\PP{\mathbb{P}}

\newcommand\RR{\mathbb{R}}

\newcommand\ZZ{\mathbb{Z}}

\newcommand\bC{\mathbf{C}}

\newcommand\bM{\mathbf{M}}

\newcommand\fC{\mathfrak{C}}

\newcommand\fM{\mathfrak{M}}

\newcommand\frq{\mathfrak{q}}

\newcommand\frt{\mathfrak{t}}

\newcommand\arr{\ifinner\to\else\longrightarrow\fi}

\def\displaytimes_#1{\mathrel{\mathop{\times}\limits_{#1}}}

\def\displayotimes_#1{\mathrel{\mathop{\bigotimes}\limits_{#1}}}

\newcommand\spec{\operatorname{Spec}}

\newcommand\id{\mathrm{id}}

\newcommand\GL{\operatorname{GL}}
\newcommand\SL{\operatorname{SL}}

\newdir{ >}{{}*!/-5pt/@{>}}

\newcommand\doublelong[2]{\mathbin{\xymatrix{{}\ar@<3pt>[r]^{#1}
\ar@<-3pt>[r]_{#2}&}}}

\newlength{\ignora}

\renewcommand{\setminus}{\smallsetminus}

\numberwithin{equation}{subsection}

\newcommand{\lM}{{M^{\dagger}}}

\newcommand{\lx}{x^{\dagger}}
\newcommand{\ly}{{y^{\dagger}}}

\newcommand{\ldelta}{{\delta^{\dagger}}}
\newcommand{\oh}{{\overline{h}}}

\newcommand{\tx}{\tilde{x}}
\newcommand{\ty}{\tilde{y}}
\newcommand{\tz}{\tilde{z}}

\newcommand{\tL}{\tilde{L}}
\newcommand{\tU}{\tilde{U}}

\newcommand{\tlx}{\tilde{x}^{\dagger}}
\newcommand{\tly}{\tilde{y}^{\dagger}}

\newcommand{\tpi}{\tilde{\pi}}

\newcommand{\upi}{\underline{\pi}}

\newcommand{\uPP}{\underline{\PP}}

\newcommand{\ueta}{\underline{\eta}}

\newcommand{\sqreta}{\sqrt{\eta}}

\newcommand{\lH}{{\cH}^{\dagger}}

\newcommand{\mH}{{\overline{\cH}}}

\newcommand{\oM}{{\overline{\bM}}}

\newcommand{\oC}{{\overline{\bC}}}

\newcommand{\ord}{\operatorname{ord}}

\newcommand{\Res}{\operatorname{Res}}

\newcommand{\Rpi}{\mathcal{R}\pi_{*}}

\newcommand{\Hyp}{\mathcal{H}yp}

\begin{document}

\title[Log twisted differentials]{Spin and hyperelliptic structures of log twisted differentials}

\author{Dawei Chen}

\author{Qile Chen}

\thanks{D. Chen is supported by the NSF CAREER grant DMS-1350396 and Q. Chen is supported by the NSF grant DMS-1560830.}

\address[Chen]{Department of Mathematics\\
Boston College\\
Chestnut Hill, MA 02467\\
USA}
\email{dawei.chen@bc.edu}

\address[Chen]{Department of Mathematics\\
Boston College\\
Chestnut Hill, MA 02467\\
USA}
\email{qile.chen@bc.edu}

\date{\today}

\begin{abstract}

Using stable log maps, we introduce log twisted differentials extending the notion of abelian differentials to the Deligne-Mumford boundary of stable curves. The moduli stack of log twisted differentials provides a compactification of the strata of abelian differentials. The open strata can have up to three connected components, due to spin and hyperelliptic structures. 
We prove that the spin parity can be distinguished on the boundary of the log compactification. Moreover, combining the techniques of log geometry and admissible covers, we introduce log twisted hyperelliptic differentials, and prove that their moduli stack provides a toroidal compactification of the hyperelliptic loci in the open strata.

\end{abstract}

\keywords{Abelian differential, log twisted differential, spin parity, hyperelliptic structure}
\subjclass[2010]{14H10 (primary), 14D23(secondary)}

\maketitle

\tableofcontents


\section{Introduction}\label{sec:intro}

Let $\mu = (m_1, \ldots, m_n)$ be a partition of $2g-2$. Consider the stratum $\cH(\mu)$ of abelian differentials with signature $\mu$ parameterizing pairs 
$(C, \omega)$, where $C$ is a smooth genus $g$ curve and $\omega$ is an abelian differential on $C$ whose zero type is specified by $\mu$. The study of $\cH(\mu)$ has broad connections to many fields. For example, an abelian differential defines a flat structure with conical singularities at its zeros. The behavior of geodesics under the flat metric is closely related to billiards in polygons, which has produced abundant results in dynamics and geometry. From the algebraic viewpoint, an abelian differential is determined, up to scaling, by its underlying canonical divisor. Hence the projections of these strata $\cH(\mu)$ to the moduli space of genus $g$ curves $\bM_g$ provide a series of special subvarieties according to properties of canonical divisors. Moreover, there is a $\GL^{+}(2, \RR)$-action on $\cH(\mu)$ induced by varying polygonal representations of abelian differentials. The closure of any 
$\GL^{+}(2, \RR)$-orbit (under the analytic topology on $\cH(\mu)$) is known to be algebraic by the seminal work of Eskin-Mirzakhani-Mohammadi (\cite{EM, EMM}) and Filip (\cite{Filip}). We refer to \cite{Zorich, MoellerTeich, Wright, Bootcamp} for surveys on abelian differentials and relevant topics. 

Just as many moduli problems, it is natural to pursue a functorial compactification for the stratum $\cH(\mu)$. Then one can obtain information for the interior of the stratum by analyzing the boundary of the compactification. This circle of ideas has been shown to be useful. For instance, dynamical invariants of Teichm\"uller curves (as dimensionally minimal $\GL^{+}(2, \RR)$-orbit closures) can be computed through their intersections with divisors of the Deligne-Mumford compactification $\oM_{g}$ (\cite{CMab, CMquad}). With this as a motivation, in this paper we consider a compactification of $\cH(\mu)$ via log geometry. 

\subsection{Moduli of log twisted sections}\label{ss:intro-log-twist-sec}
Our construction of the compactification is based on the theory of stable log maps developed in \cite{GS, Chen, AC}. We first treat a more general situation. Let $\uL$ be a line bundle of relative degree $d$ over a family of genus $g$ pre-stable curves $\upi: \uC \to \uM$ with $n$ markings $\sigma_1,\ldots, \sigma_n$, and $\mu = (m_1, \ldots, m_n)$ a partition of $d$. For each $\uS \to \uM$, consider a family of log curves $\pi_S: C_S \to S$  over the underlying curve $\uC_{\uS} \to \uS$ obtained by pulling back the family $ \uC \to \uM$. See Section \ref{ss:log-curve} for the definition of log curves. We have a projection of log schemes $\psi_S: L_{S} \to C_S$ such that  
\begin{enumerate}
 \item The underlying structure $\uL_S$ is the total space of the pull-back of $\uL$ via $\uC_{\uS} \to \uC$;
 \item The log structure $\cM_{L_S}$ on $L_S$ is given by the zero section of the line bundle $\uL_S$ together with the log structure from $C$, see Section \ref{ss:log-section}. 
\end{enumerate}
A {\em log twisted section with signature $\mu$ over the log scheme $S$} is a section $\eta: C_S \to L_S$ of the projection $\psi_S$ with contact orders at the marked points given by $\mu$. Note that the underlying morphism $\underline{\eta}$ is a section of the line bundle $\uL_S$ over $\uC_{\uS}$. Here the contact order is a generalization of the vanishing order of sections of the line bundle $\uL$. It remembers the vanishing order at the marked points, even when the underlying section $\underline{\eta}$ vanishes completely along the components of the curve containing the marked points, see (\ref{equ:marked-pt}).

Denote by $\lM(\mu)$ the category of log twisted sections fibered over the category of log schemes. Note that algebraic stacks are category fibered over schemes rather than log schemes. In order to prove the representability of $\lM(\mu)$ as a Deligne-Mumford stack, we use the key observation discovered in \cite{GS, Chen, AC} that one can view $\lM(\mu)$ as a fibered category of minimal objects over the category of schemes. We discuss the minimality of log twisted sections in Section \ref{ss:minimality} and show that
\begin{thm}[Theorem \ref{thm:log-twisted-sec}]
The fibered category $\lM(\mu)$ is represented by an algebraic stack with its minimal log structure which  parameterizes minimal log twisted sections. Furthermore, the forgetful morphism 
\[
\lM(\mu) \to \Rpi(\uL)
\]
is representable and finite, where $\Rpi(\uL)$ is the total space of the push-forward of $\uL$ along $\upi$, see Section \ref{ss:log-section}.
\end{thm}

\subsection{Log twisted differentials and their spin parity}\label{ss:intro-log-twisted-diff}
Consider the situation when $\uC \to \uM$ is given by the universal family of genus $g$, $n$-marked stable curves $\oC_{g,n} \to \oM_{g,n}$ with $\uL = \omega_{\oC_{g,n}/\oM_{g,n}}$ the (relative) dualizing line bundle. Let $\mu$ be a partition of $2g - 2$, which equals the degree of $\omega_{\oC_{g,n}/\oM_{g,n}}$. A {\em log twisted differential with signature $\mu$} is a log twisted section with signature $\mu$ with respect to the dualizing line bundle, see Section \ref{ss:log-twist-diff}. We emphasize that log twisted differential extends the notion of abelian differential functorially to the Deligne-Mumford boundary. It naturally associates to each irreducible component of a degenerate curve a (possibly meromorphic) differential up to a $\CC^*$-scaling controlled by the log structures, see Section \ref{ss:fiberwise-induced-diff}.

Denote by $\lH(\mu)$ the moduli stack of log twisted differentials with signature $\mu$. The open stratum $\cH(\mu) \subset \lH(\mu)$ is the locus parameterizing (not identically zero) abelian differentials over smooth curves. In particular, $\cH(\mu)$ can be identified with the open locus of $\lH(\mu)$ with the trivial log structure.

For a partition $\mu$ of $2g-2$ with even entries $m_i = 2k_i$ for all $i$, the half-canonical divisor $\sum_{i=1}^n k_i\sigma_i $ given by an abelian differential over a smooth curve defines a theta characteristic, called the {\em spin structure}. The parity $\dim H^0(\sum_{i=1}^n k_i\sigma_i ) \pmod{2}$ of the spin structure is a smooth deformation invariant (\cite{Atiyah, Mumford}). It follows that the loci of abelian differentials with odd and even spin structures in the stratum $\cH(\mu)$ are disjoint. Kontsevich and Zorich (\cite{KZ}) classified the connected components for all strata $\cH(\mu)$. It turns out that $\cH(\mu)$ can have up to three connected components, distinguished by spin and hyperelliptic structures. We next generalize the spin parity to log twisted differentials and prove the following result. 

\begin{thm}[Theorem~\ref{thm:spin}]
Consider the partition $\mu = (m_{i})_{i=1}^{n}$ with each $m_{i}$ even. Then we have a disjoint union 
\[
\lH(\mu) = \lH(\mu)^{+}\sqcup \lH(\mu)^{-},
\]
where $\lH(\mu)^{+}$ (resp. $\lH(\mu)^{-}$) is the open and closed substack of $\lH(\mu) $ parameterizing log twisted differentials with even (resp. odd) spin parity.
\end{thm}

\begin{remark}
We remark that the moduli space $\lH(\mu)$ of log twisted differentials maps onto the Farkas-Pandharipande moduli space of twisted canonical divisors (\cite{FP}), hence it may contain components that are entirely over the Deligne-Mumford boundary. Nevertheless, the spin parity defined above is not only for the main component of $\lH(\mu)$, but also for all boundary components that contain non-smoothable log twisted differentials. Therefore, $\lH(\mu)^{+}$ and $\lH(\mu)^{-}$ are in general reducible, consisting of other components besides the main (spin) components. In addition, depending on the spin parity, one of $\lH(\mu)^{+}$ and $\lH(\mu)^{-}$ contains the hyperelliptic component if it exists. 
\end{remark}

\subsection{The hyperelliptic structure}
An abelian differential $\eta$ over a smooth curve $C$ is called {\em hyperelliptic} if $C$ is hyperelliptic and $\iota^{*}\eta = - \eta$, where $\iota$ is the hyperelliptic involution. In order to remember the involution structure over the boundary, we combine the techniques of log geometry and admissible covers, and define {\em log twisted hyperelliptic differentials}, see Section \ref{ss:hyp-diff}. Given a partition $\mu$ of $2g - 2$ compatible with the hyperelliptic involution, we introduce the category of log twisted hyperelliptic differentials fibered over the category of log schemes, denoted by $\Hyp(\mu)$, and prove that

\begin{thm}[Theorem \ref{thm:hyp}]
The fibered category $\Hyp(\mu)$ is represented by a separated, log smooth Deligne-Mumford stack with its universal minimal log structure. Furthermore, the forgetful morphism to the Hodge bundle $\Hyp(\mu) \to \cH_{g,n}$ is representable and finite.
\end{thm}

\begin{remark}
For the convenience of the construction, in the hyperelliptic case we mark all the involution fixed points, hence zero contact orders are allowed in the partition $\mu$, see Section \ref{ss:hyp-moduli}. 
 \end{remark}

\begin{remark} 
The compactification $\Hyp(\mu)$ treats not only the hyperelliptic components of $\cH(\mu)$, but also the hyperelliptic loci in $\cH(\mu)$ for any partition $\mu$ compatible with $\iota$. The log smoothness of $\Hyp(\mu)$ is equivalent to the statement that the boundary of $\Hyp(\mu)$ is toroidal. In particular, all log twisted hyperelliptic differentials are smoothable. 
\end{remark}

\subsection{Related works and comparison}\label{ss:related}

There are several recent attempts to study degeneration of abelian differentials from the classical viewpoint without using log structures, i.e., compactifying the strata in the Hodge bundle over $\oM_{g,n}$ (resp. in $\oM_{g,n}$) using twisted differentials (resp. twisted canonical divisors). Gendron (\cite{Gendron}) proved the smoothability of twisted differentials when they have zero residues at all nodes. The first author (\cite{ChenDiff}) studied twisted canonical divisors on curves of pseudo-compact type. Farkas and Pandharipande (\cite{FP}) studied systematically the space of twisted canonical divisors and showed that it is in general reducible, which contains, besides the closure of the stratum, a number of boundary components that have dimension one less than the dimension of the stratum. Finally in \cite{BCGGM1} a crucial global residue condition was found and used to isolate the closure of the stratum. As mentioned before, the moduli space $\lH(\mu)$ of log twisted differentials maps onto the Farkas-Pandharipande space of twisted canonical divisors, hence it also contains extra boundary components. Nevertheless, there is a way to implant the global residue condition under the log setting to isolate the main component, which will be addressed in our future work. 

Without log structures, the odd and even spin components can intersect in the boundary of the aforementioned compactifications. This phenomenon has been already observed in \cite{EMZ, Gendron, ChenDiff}. 

We mention that the log twisted differential considered in this paper is closely related to the strata compactification studied by Gu\'{e}r\'{e} (\cite{Guere}), whose construction also relies on the theory of stable log maps, but is motivated from the viewpoint of Gromov-Witten theory. 

If one replaces the canonical line bundle by its $k$-th power, then it is natural to consider the strata of $k$-differentials. Setting $L$ to be the $k$-th power of the relative canonical line bundle in Theorem~\ref{thm:log-twisted-sec}, our result gives also a log compactification for the strata of $k$-differentials parameterizing log twisted $k$-differentials (similarly see \cite{Guere}). The space of twisted $k$-differentials is in general reducible and a modified global residue condition via canonical $k$-covers can isolate the main component (\cite{BCGGM3}). 

\subsection{Conventions and notations}

We use capital letters such as $C$, $S$, and $M$ to denote log schemes or log stacks. Their underlying schemes or stacks are denoted by $\uC$, $\uS$, and $\uM$ respectively. Given a log structure $\cM_{X}$ over a scheme $\uX$, denote by $\ocM_{X} := \cM_{X}/\cO^*$ the {\em characteristic sheaf} of $\cM_{X}$. All log structures considered here are fine and saturated.  We refer to \cite{KKato, log-survey} for the basics of log geometry. 

This paper is heavily built on the theory of stable log maps, and we refer to \cite{AC, Chen, GS} for general developments of stable log maps. Throughout this paper, we work over an algebraically closed field of characteristic zero. 

\subsection*{Acknowledgements} 

The authors thank Dan Abramovich, Matt Bainbridge, Gavril Farkas, Quentin Gendron, Samuel Grushevsky, J\'{e}r\'{e}my Gu\'{e}r\'{e}, Martin M\"oller, Rahul Pandharipande and Jonathan Wise for stimulating discussions on related topics. 

\section{Log twisted sections}\label{sec:log-section}

In this section, we introduce the general set-up of log twisted sections of a given line bundle, and study their moduli stacks. 

\subsection{Log curves}\label{ss:log-curve}
First recall the canonical log structure on pre-stable curves (\cite{FKato, LogCurve}). 
Denote by $\underline{\fC}_{g,n} \to \underline{\fM}_{g,n}$ the universal family of pre-stable curves. Note that the boundary $\Delta_{\fM} \subset \underline{\fM}_{g,n}$ parameterizing singular underlying curves is a normal crossings divisor in $\underline{\fM}_{g,n}$. Thus we denote by $\fM_{g,n}$ the log stack, with the underlying structure given by $\underline{\fM}_{g,n}$ and the log structure $\cM_{\fM_{g,n}}$ given by the divisorial log structure associated to $\Delta_{\fM}$, see \cite[(1.5)]{KKato}

Similarly, we consider the boundary $\Delta_{\fC} \subset \underline{\fC}_{g,n}$ consisting of singular fibers with $n$ markings. The boundary $\Delta_{\fC}$ is again a normal crossings divisor in $\underline{\fC}_{g,n}$. Denote by $\fC_{g,n}$ the log stack with underlying structure $\underline{\fC}_{g,n}$, and by $\cM_{\fC_{g,n}}$ the log structure of $\fC_{g,n}$ given by the divisorial log structure associated to the boundary $\Delta_{\fC}$. Since the morphism of the pairs 
\[
(\underline{\fC}_{g,n}, \Delta_{\fC}) \to (\underline{\fM}_{g,n}, \Delta_{\fM})
\]
is toroidal, it induces a morphism of log stacks
\[
\fC_{g,n} \to \fM_{g,n}.
\]

Now consider a family of genus $g$, $n$-marked pre-stable curves $\uC \to \uS$. Such a family is obtained from the following cartesian diagram 
\[
\xymatrix{
\uC \ar[r] \ar[d] & \underline{\fC}_{g,n} \ar[d] \\
\uS \ar[r] & \underline{\fM}_{g,n}
}
\]
via a unique morphism $\uS \to \underline{\fM}_{g,n}$ determined by the underlying family $\uC \to \uS$. Pulling back the log structures on $\fC_{g,n}$ and $\fM_{g,n}$, we obtain a family of log schemes 
$$C^{\sharp} \to S^{\sharp}.$$ 
Such a family $C^{\sharp} \to S^{\sharp}$ is called the {\em basic log curve} associated to the underlying family $\uC \to \uS$. By \cite{FKato, LogCurve}, the log structure of the basic log curve is uniquely determined by the underlying family $\uC \to \uS$, up to isomorphism.  

\begin{definition}\label{def:log-curve}
A {\em genus $g$, $n$-marked log curve} over a log scheme (or a log stack) $S$ is a family $C \to S$ such that
\begin{enumerate}
 \item The underlying family $\uC \to \uS$ is a family of genus $g$, $n$-marked pre-stable curves;
 \item There is a morphism $S \to S^{\sharp}$ with the identity underlying morphism that fits in a cartesian diagram of log schemes:
 \[
 \xymatrix{
 C \ar[r] \ar[d] & C^{\sharp} \ar[d] \\
 S \ar[r]^{h} & S^{\sharp} \\
 }
 \]
 where $C^{\sharp} \to S^{\sharp}$ is the basic log curve associated to the underlying family $\uC \to \uS$. 
\end{enumerate} 
\end{definition}

We next recall the local description of a log curve over a geometric point. Given a family of log curves $\pi: C \to S$ with $\uS$ a geometric point. We denote by $\sigma_{i}: \uS \to \uC$ the $i$-th marked point. For ease of notation, we also write $\sigma_{i}$ for its image in $\uC$, and apply the same philosophy to similar notations later on. Note that there are three cases for points in a marked nodal curve: smooth unmarked points, markings, and nodes. 

\vspace{2mm}

\noindent
{\bf Smooth unmarked points.}
Let $p \in C$ be a smooth unmarked point. Then we have
\[
\cM_{C,p} = (\pi^*\cM_{S})_{p}.
\]

\vspace{2mm}

\noindent
{\bf Marked points.}
For the markings, denote by $\cM_{i}$ the divisorial log structure on $\uC$ associated to the divisor $\sigma_{i}$. We thus have the fiber
\[
\cM_{C,\sigma_{i}} = (\pi^*\cM_{S}\oplus_{\cO^*_{C}}\cM_{i})_{p}.
\]
In particular, the fiber of the characteristic sheaf at $\sigma_i$ is
\[
\ocM_{C,\sigma_{i}} = \pi^*\ocM_{S}\oplus \NN.
\]

\vspace{2mm}

\noindent
{\bf Nodes.}
We recall that $\ocM_{S^{\sharp}} = \NN^{k}$ where $k$ is the number of nodes of $\uC$. Indeed, let $\uG$ be the dual graph of $\uC$, where vertices of $\uG$ correspond to the irreducible components of $\uC$, edges of $\uG$ correspond to the nodes, and rays correspond to the markings. The generators of $\ocM_{S^{\sharp}}$ can be labeled by the set of edges $E(\uG)$. For each edge $l \in E(\uG)$, denote by $e_{l} \in \ocM_{S^{\sharp}}$ the corresponding generator. Suppose $l$ represents the node $z$. Then we will also set $e_z = e_l$ and use them interchangeably. Let $x$ and $y$ be local coordinates of the two components meeting at the node $z$. Then we can lift the two functions $x$ and $y$ to two sections $\lx$ and $\ly$ in $\cM_{C^{\sharp}}$ \'etale locally near $z$. By a careful choice of the chart $\beta^{\sharp}: \ocM_{S^{\sharp}} \to \cM_{S^{\sharp}} $ and the coordinates $x$ and $y$, we have the following local equation
\begin{equation}\label{equ:node-smoothing}
\beta^{\sharp}(e_{l}) = \lx + \ly\quad \mbox{in} \ \cM_{C^{\sharp}}. 
\end{equation}
We have the fiber 
\[
\ocM_{C^{\sharp},l} = \NN^{k-1}\oplus\NN^{2},
\]
where the first summand $\NN^{k-1}$ is generated by $\{e_{l'}\}$ for edges $l' \neq l$, and the second summand $\NN^{2}$ is generated by the images of $\lx$ and $\ly$. 
 
We denote again by $\lx$ and $\ly$ the pull-backs of the corresponding local sections over $C$. Thus locally near $l$, we have 
\[
\oh^{\flat}\circ \beta^{\sharp}(e_{l}) = \lx + \ly\quad \mbox{in} \ \ocM_{C}, 
\]
where $\oh^{\flat}$ is the morphism on the level of characteristic sheaves induced by the map $h$ in Definition~\ref{def:log-curve}. We have the fiber
\[
\ocM_{C,l} = \pi^*\ocM_{S}\oplus_{\ocM_{S^{\sharp}}}\NN^{2},
\]
where the second summand $\NN^{2}$ is generated by the images of $\lx$ and $\ly$. For later use, we will identify $e_{l}$ (or $e_z$) with its image in $\ocM_{S}$, and call it the {\em smoothing element} associated to the edge $l$ (or the node $z$). 

\subsection{Log twisted sections}\label{ss:log-section}
Here we study log twisted sections using stable log maps developed in (\cite{GS, Chen, AC}). 
Throughout the rest of this section, we consider the following situation: 
\begin{enumerate}
 \item $\upi: \uC \to \uM$ is a family of genus $g$, $n$-marked pre-stable curves over an algebraic stack $\uM$ with the $n$ markings $ \{\sigma_{i}\}_{i=1}^{n}$;
 \item $\uL$ is a line bundle of degree $d$ over the family of curves $\uC$. 
\end{enumerate}
For any morphism $\uS \to \uM$, denote by $(\uL_{\uS}, \uC_{\uS})$ the family obtained by pulling back $(\uL, \uC)$.

Consider the stack $\Rpi(\uL)$ over $\uM$ that associates to each morphism $\uS \to \uM$ the set of global sections $H^0(\uL_{\uS})$. The stack $\Rpi(\uL)$ is algebraic with a separated and representable morphism $\Rpi(\uL) \to \uM$ (\cite{CL}).

We again write $\pi: C \to M$ for the basic log curve associated to the underlying family $\upi: \uC \to \uM$ with markings $\{\sigma_{i}\}_{i=1}^{n}$. Let $\underline{\psi}: \uL \to \uC$ be the projection. Let $\cN'_{L}$ be the divisorial log structure on the total space of $\uL$ associated to the zero section of $\uL$, and $\cN_{L} = \psi^*\cM_{C}$ the log structure on $\uL$. Define 
\begin{equation}\label{eq:log-lift} 
\cM_{L} := \cN'_{L}\oplus_{\cO^*}\cN_{L}, 
\end{equation}
and denote by $L = (\uL, \cM_{L})$ the log stack. We thus obtain a log smooth morphism
\[
\psi: L \to C
\]
whose underlying morphism is given by $\underline{\psi}$. 

Consider a partition
\[
\mu = (m_{1},\cdots,m_{n})
\]
such that $m_{i} \in \NN$ and $\sum_{i=1}^{n} m_i = d$. Here we allow $m_i = 0$. Our motivation is to compactify the stratum of global sections of $\uL$ that have a zero of order $m_i$ at each $\sigma_i$. Now we formally define such sections under the log geometric setting. 

\begin{definition}\label{def:log-twisted-sec}
A family of {\em log twisted sections with contact orders $\mu$ associated to the family $\uL \to \uC \to \uM$} over a log scheme $S$ consists of the following data: 
\begin{equation*}
(\pi_{S}: C_{S} \to S, \ \eta: C_{S} \to L_{S})
\end{equation*}
where 
\begin{enumerate}
 \item $\pi_{S}: C_{S} \to S$ is the family of log curves obtained by pulling back the family $C \to M$ via the map $S \to M$;
 \item The composition $\psi\circ\eta$ is the identity map of the log curve $C_{S}$. In particular, the underlying map $\ueta$ is a global section of the line bundle $\uL_{S}$ over $\uC_{S}$; 
 \item The contact order at the $i$-th marking is $m_{i}$ for $i=1, \cdots, n$ (see Section~\ref{ss:local-structure} for the precise definition of contact order). 
\end{enumerate}
\end{definition}
For simplicity, we will use log twisted section without mentioning the line bundle $\uL$ and the partition $\mu$ when there is no confusion. Sometimes we also refer to $\eta$ as a log twisted section if the family of log curves $C_{S} \to S$ is clear in the context.  

Denote by $L' = (\uL, \cN'_{L})$ with the divisorial log structure associated to the zero section of $\uL$. Then we have the following observation that every log twisted section $\eta$ is automatically a stable log map to $L'$. 

\begin{lemma}\label{lem:sec-map}
A log twisted section $\eta$ is equivalent to a stable log map $\eta': C_{S} \to L'_{S}$ with the underlying morphism given by a section $\ueta \in H^0(\uL)$. We call $\eta'$ the {\em stable log map associated to $\eta$}. 
\end{lemma}
\begin{proof}
On the level of log structures, the morphism $\eta^{\flat}: \cN'_{L_{S}}\oplus_{\cO^*}\cN_{L_{S}} \to \cN_{L_{S}}$ induces an isomorphism on the component $\cN_{L_{S}} $ as it is a section of $\psi_{S}$. 
\end{proof}

We denote by $\lM(\mu)$ the category of log twisted sections fibered over the category of log schemes. The main result of this section is 

\begin{thm}\label{thm:log-twisted-sec}
The fibered category $\lM(\mu)$ is represented by an algebraic stack with the minimal log structure. Furthermore, the morphism 
\[
\lM(\mu) \to \Rpi(\uL)
\]
is representable and finite.
\end{thm}

The finiteness of the theorem follows from the same property of stable log maps, see \cite{GS, Chen, AC}. We will justify the minimal log structure in Section~\ref{ss:minimality}, and the representability of $\lM(\mu)$ will follow from Proposition~\ref{prop:representability}. 

\subsection{Local structure of log twisted sections}\label{ss:local-structure}

Let $\eta$ be a log twisted section over a log curve $\pi_{S}: C_{S} \to S$. Denote by $\eta'$ the stable log map associated to $\eta$. Shrinking $S$ if necessary, there exists a chart 
\[
\beta: \ocM_{S,s} \to \cM_{S}
\]
for some point $s \in S$. For an element $e \in \ocM_{S,s}$, we identify $e$ with its image $\beta(e)$ for the above choice of chart.

Consider the case when $\uS$ is a geometric point. Then the chart $\beta$ becomes a section of the natural quotient morphism
\[
\cM_{S} \to \ocM_{S} := \cM_{S}/\cO^*_{S}.
\]
We want to analyze three cases: irreducible components of $C$, markings, and nodes. 

\vspace{2mm}
\noindent
{\bf Irreducible components.}
Let $Z$ be an irreducible component of $C$. Locally near a smooth unmarked point $z \in Z$, we have the morphism on the level of log structures
\begin{equation}\label{equ:general-pt}
(\eta')^{\flat}: \eta^*\cN'_{L_{S},z} \to \cM_{C, z} := \pi_{S}^*\cM_{S}, \quad \ldelta \mapsto e_{Z} + u
\end{equation}
where $\ldelta$ is a local generator of $\cN'_{L_{S},z}$ near $\eta(z)$, $e_{Z}$ is the section $\beta(e_{Z}) \in \cM_{S}$ by a slight abuse of notation, and $u \in \cO_{C}$ is an invertible regular function in a neighborhood of $z$. Note that if $e_{Z} \neq 0$ in $\ocM_{S}$, then the underlying section $\underline{\eta}$ vanishes entirely on the component $Z$. Hence we call $e_{Z} \in \ocM_{S}$ the {\em degeneracy} of the irreducible component $Z$. 

\vspace{2mm}
\noindent
{\bf Marked points.}
For a marking $z = \sigma_{i} \in Z$, we have the morphism on the level of log structures
\begin{equation}\label{equ:marked-pt}
(\eta')^{\flat}: \eta^*\cN'_{L_{S},z} \to \cM_{C_{S},z} := \pi_{S}^*\cM_{S}\oplus_{\cO^*}\cN_{\sigma_i}, \quad \ldelta \mapsto e_Z + m_i \cdot \delta_i^{\dagger} + u
\end{equation}
where $e_Z$ is the degeneracy of $Z$ defined above, $\delta_i^{\dagger} \in \cM_{C_{S}}$ is the lifting of the local defining equation $\delta_{i}$ of $\sigma_{i}$, and $u \in \cO_{C}$ is an invertible regular function in a neighborhood of $\sigma_i$. 

It is important to notice that even if $e_{Z} \neq 0$ and consequently the underlying section $\underline{\eta}$ vanishes on $Z$, the morphism on the level of log structures remembers the zero orders at the markings. 

\begin{definition}
We call $m_{i}$ in the above the {\em contact order} of $\eta$ at $\sigma_{i}$. 
\end{definition}
By \cite{Chen, GS}, the contact order remains constant in a connected family of stable log maps. 

\vspace{2mm}
\noindent
{\bf Nodes.}
Let $z \in C$ be a node joining two irreducible components $Z_1$ and $Z_2$. Let $l$ be the edge in the dual graph $\uG$ of $\uC$ that corresponds to the node $z$. As said we will interchangeably use both $l$ and $z$ for the same object. 
Let $\lx$ and $\ly$ be local sections in $\cM_{C}$ near $z$ as in (\ref{equ:node-smoothing}) corresponding to the two components $Z_{1}$ and $Z_{2}$ respectively. Without loss of generality we can assume on the level of log structures
\begin{equation}\label{equ:node}
(\eta')^{\flat}: \eta^*\cN'_{L_{S}} \to \cM_{C,z}, \quad \ldelta \mapsto e_{Z_1} + c_{l}\cdot\lx + u 
\end{equation}
where $c_{l}$ is a non-negative integer, $e_{Z_1} = \beta(e_{Z_1})$ again by a slight abuse of notation, and $u$ is some invertible function in a neighborhood of $z \in Z_1$. Note that in this case there is a relation of the degeneracies of $Z_1$ and $Z_2$ on the level of characteristic monoid: 
\begin{equation}\label{equ:node-deg}
e_{Z_1} + c_{l} \cdot  e_{z} = e_{Z_2}. 
\end{equation}
Let $V \subset Z_{2}$ be a small neighborhood of $z$. Then on $V\setminus z$ we have
\begin{equation}\label{equ:sec-pole}
(\eta')^{\flat}(\ldelta) = e_{Z_2} - c_l \cdot \ly + v  
\end{equation}
for some invertible function $v \in \cO_{V}$. 

\begin{definition}\label{def:node-contact}
In the above setting, we say that $c_{l}$ (or $c_{z}$) is the {\em contact order} of $\eta$ at the node $z$ (or the edge $l$).
\end{definition}
It is useful to notice that when $Z_1 = Z_2$, using the relation (\ref{equ:node-deg}), we have $c_l = 0$.

\vspace{2mm}
Note that the relation~\eqref{equ:node-deg} implies a {\em partial ordering} ``$\geq_{\min}$'' on the set $V(\uG)$ of irreducible components, which we call the {\em minimal partial ordering}. More precisely, we write $Z_1 >_{\min} Z_2$ if $c_l > 0$ and $Z_1 =_{\min} Z_2$ if $c_l = 0$.

\begin{remark}
The idea is that $e_Z$ measures the degeneracy of $Z$. If $c_l > 0$, then the log twisted section $\eta$ degenerates faster along $Z_2$ than $Z_1$, and hence $Z_2$ looks ``smaller'' compared to $Z_1$. Similarly for $c_l = 0$, it means $\eta$ has the same degenerate speed along $Z_1$ and $Z_2$, hence they have relatively comparable ``sizes''. 

\end{remark}



\subsection{Minimality}\label{ss:minimality}

We discuss the minimality of stable log maps in the situation of log twisted sections. We refer to \cite{Chen, AC, GS} for more general situations. First assume that $\eta$ is a log twisted section over a log curve $C_{S} \to S$ with $\uS$ a geometric point, as considered in the preceding section. Let $\eta'$ be the stable log map as in Lemma~\ref{lem:sec-map}. 

Denote by $\uG$ the dual graph of the underlying curve $\uC_{S}$. Recall that $E(\uG)$ is the set of edges corresponding to nodes of $\uC_S$ and $V(\uG)$ is the set of vertices corresponding to irreducible components of $\uC_S$. We then define the {\em weighted graph} $G$ associated to $\eta$, which is formed out of the dual graph $\uG$ with the following extra data:
\begin{enumerate}
 \item The contact order $c_{l}$ for each edge $l \in E(G)$;
 \item A partition $V(G) = V^d(G)\sqcup V^{nd}(G)$ such that $e_{v} \neq 0$ if and only if $v \in V^d(G)$, where $e_{v} = e_Z$ is the degeneracy of the irreducible component $Z$ corresponding to $v$;
 \item The minimal partial ordering ``$\geq_{\min}$'' on $V(G)$ defined in Section~\ref{ss:local-structure}. 
\end{enumerate}

In the above, we wrote $V(G)$ instead of $V(\uG)$ etc for the part of the data from the underlying dual graph. The decorations ``$d$'' and ``$nd$'' on the partition of $V(G)$ indicate whether or not a component $Z$ is degenerate. In other words, it is degenerate if and only if the degeneracy of $\eta$ at $Z$ is non-trivial, that is, if and only if the underlying section $\ueta$ vanishes entirely on $Z$. 
For an edge joining two vertices $v_1 \geq_{\min} v_2$, we call the corresponding node an \emph{incoming node} of $v_2$. In other words, the orientation of an edge is from the bigger vertex to the smaller one. 

Consider the free abelian group 
\[
\langle e_v, e_l \rangle_{v \in V(G),\ l \in E(G)}
\]
generated by the degeneracy and smoothing elements, Take the quotient
\[
\langle e_v, e_l \rangle_{v \in V(G),\ l \in E(G)} / \sim
\]
where the relation ``$\sim$'' is given by 
\begin{enumerate}
 \item (non-degeneracy) $e_v \sim 0$ if $v \in V^{nd}(G)$;
 \item (edge relation) if $l$ is an edge joining $v_{1}$ and $v_{2}$ with $v_{1} \geq_{\min} v_{2}$, then 
\begin{equation}\label{equ:edge-relation}
e_{v_{1}} + c_{l} \cdot e_l \sim e_{v_{2}}.
\end{equation}
\end{enumerate}
Consider the maximal torsion-free subgroup
\[
\cG \subset \langle e_v, e_l \rangle_{v \in V(G),\ l \in E(G)} / \sim.
\]
We introduce the {\em minimal monoid} $\cM(G) \subset \cG$ which is the saturation submonoid in $\cG$ generated by the images of $e_v$ and $e_l$ for all $v \in V(G)$ and $l \in E(G)$. By \cite{Chen, AC, GS}, we have
\begin{enumerate}
 \item $\ocM(G)$ is a fine, saturated, and sharp monoid;
 \item $\ocM(G)^{gp} = \cG$;
 \item there is a canonical morphism 
\begin{equation}\label{equ:minimal}
\psi: \ocM(G) \to \ocM_{S}.
\end{equation}
\end{enumerate}

\begin{definition}\label{def:minimal}
The log twisted section $\eta$ is called {\em minimal} if $\psi$ is an isomorphism. A family of log twisted sections is called {\em minimal} if each geometric fiber is minimal.
\end{definition}

By \cite{Chen, AC, GS}, minimality forms an open condition over a family of log twisted sections, hence the above definition of minimality works well with families. Furthermore, we have the following universal property, see \cite{Chen, AC, GS} for a proof. 

\begin{proposition}\label{prop:minimal}
For any log twisted sections $\eta$ over a family of log curves $C_{S} \to S$, there is a minimal family of log twisted sections $\eta_{m}$ over a family of log curves $C_{S_{m}} \to S_{m}$ and a morphism $S \to S_{m}$ whose underlying map is the identity, such that $\eta$ is the pull-back of $\eta_{m}$ via $S \to S_{m}$. Moreover, the pair $(\eta_{m}, S \to S_m)$ is unique up to isomorphism. 
\end{proposition}

Recall that $\lM$ is the category of log twisted sections fibered over the category of log schemes. 
Denote by $M_{m}^{\dagger}$ the category of minimal log twisted sections fibered over the category of schemes. As a fibered category, $M_{m}^{\dagger}$ carries a universal minimal log structure $\cM_{M_{m}^{\dagger}}$. Hence we can view the pair $(M_{m}^{\dagger}, \cM_{M_{m}^{\dagger}})$ as a category fibered over log schemes. By the universal property of minimality in Proposition \ref{prop:minimal}, the pair $(M_{m}^{\dagger}, \cM_{M_{m}^{\dagger}})$ represents the fibered category $\lM$. Therefore, the representability of $\lM$ is a consequence of the following result. 

\begin{proposition}\label{prop:representability}
The fibered category $M_{m}^{\dagger}$ is represented by an algebraic stack. 
\end{proposition}
\begin{proof}
By \cite[Lemma C.5]{AOV}, replacing $\uM$ by a smooth cover, it suffices to consider the case when $\uC \to \uM$ is a morphism of schemes, and then $L'$ is a log scheme over the underlying scheme $\uM$. By \cite{Chen, AC, GS}, the category of minimal stable log maps to $L'$ with genus $g$, curve class given by the zero section, and contact order $\mu$ form a log algebraic stack $\fM(L')$ with its minimal log structure. There is a tautological morphism given by the underlying universal curve 
\[
\fM(L') \to \fM_{g,n}. 
\]
Consider the base change 
\[
\fM_{M} = \fM(L')\times_{\fM_{g,n}}M.
\]
It follows that $\fM_{M}$ is also algebraic. Furthermore, there is an open sub-stack $\fM^{\circ}_{M} \subset \fM_{M}$ whose underlying morphism is a section of $\uL \to \uC$. By Lemma \ref{lem:sec-map}, we have $M_{m}^{\dagger} = \fM^{\circ}_{M}$. This finishes the proof.
\end{proof}

Combining Propositions~\ref{prop:minimal} and~\ref{prop:representability} completes the proof of Theorem~\ref{thm:log-twisted-sec}.

\section{Log twisted differentials}

In this section we specialize the general theory of log twisted sections established in the preceding section to the case of differentials. 

\subsection{Log twisted differentials and their moduli}\label{ss:log-twist-diff}

Using the notations in Section~\ref{ss:log-section}, we consider the following situation: 
\begin{enumerate}
 \item $\uM = \underline{\bM}_{g,n} $ the stack of stable curves of genus $g$ with $n$ markings; 
 \item $\uC = \underline{\bC}_{g,n} $ the universal curve over $\uM$ with the projection $\upi: \uC \to \uM$; 
 \item $\uL =\omega_{\upi}$ the dualizing line bundle over $\uC$ associated to the family $\upi$.
\end{enumerate}
We thus have the family of log curves  $\pi: C \to M$ with the canonical log structure associated to the underlying family $\upi$. Over each fiber of $\upi$, the line bundle $\uL$ is of degree $d = 2g-2$. 

Denote by $\mu = (m_{1},\cdots, m_{n})$ a partition of $2g-2$. Recall that we have $L = (\uL, \cM_{L})$, with 
$\cM_{L} = \cN'_{L}\oplus_{\cO^*}\cN_{L}$ and a log smooth morphism
\[
\psi: L \to C. 
\]
Recall also that we set $L' = (\uL, \cN'_{L})$. 

For any genus $g$, $n$-marked stable log curve $\pi_{S}: C_{S} \to S$, denote by $L_{S}$ the pull-back of $L$ via the morphism $S \to M$. Now we define log twisted differential as a special case of log twisted section. 

\begin{definition}\label{def:log-twisted-diff}
A {\em log twisted differential with signature $\mu$} over a family of stable log curves $C_{S} \to S$ is a log twisted section 
$\eta: C \to L$ with contact orders $\mu$ as in Definition \ref{def:log-twisted-sec}. It is called {\em minimal} if it is minimal in the sense of Definition \ref{def:minimal}.
\end{definition}

Note that $\mu$ already includes the information of genus and number of markings. Denote by $\lH(\mu)$ the category of log twisted differentials with siganture $\mu$ over the category of log schemes. Let $\cH$ be the Hodge bundle over $\uM$ in the classical sense. Theorem~\ref{thm:log-twisted-sec} then specializes as follows. 

\begin{corollary}\label{cor:diff-moduli}
The category $\lH(\mu)$ is represented by a separated Deligne-Mumford stack of finite type with the minimal log structure. Furthermore, the tautological morphism
\[
\lH(\mu) \to \cH
\]
by removing all log structures is representable and finite. 
\end{corollary}

Note that the underlying stack of $\lH(\mu)$ is the moduli of minimal log twisted differentials. 

Let $\cH(\mu)\subset \lH(\mu)$ be the open substack with the trivial log structure. Then $\cH(\mu)$ is the moduli space of (not identically zero) abelian differentials on smooth curves whose zero orders are of type $\mu$, i.e., the stratum of abelian differentials with signature $\mu$. 
Denote by $\mH(\mu)$ the component of $\lH(\mu)$ containing $\cH(\mu)$ as an open dense substack. We call $\mH(\mu)$ the {\em main component} of $\lH(\mu)$. 

\subsection{Induced differentials on irreducible components}\label{ss:fiberwise-induced-diff}

Consider a log twisted differential $\eta$ over a family of log curves $\pi_{S}: C_{S} \to S$. We assume that $\uS$ is a geometric point. Let 
$G$ be the weighted graph of $\eta$ as defined in Section~\ref{ss:minimality}. Fix a chart $\beta: \ocM_{S} \to \cM_{S}$. Then we obtain a composition 
\[
\ocM(G) \to \ocM_{S} \to \cM_{S}
\]
again denoted by $\beta$. For each element $e \in \ocM(G)$, we identify $e$ with its image in $\cM$ under $\beta$. 

For a vertex $v \in V(G)$, let $Z \subset C_{S}$ be the corresponding irreducible component. We first consider the composition
\begin{equation}\label{equ:diff-comp}
\eta_{Z}'' := (\eta')^{\flat} - e_{v}: \cN_{L'_{S}}|_{Z} \to \cM_{C}|_{Z} \to \cO_{Z},
\end{equation}
where the first arrow is defined by $\delta \mapsto (\eta')^{\flat}(\delta) - e_{v}$ with $e_{v} = e_Z$ the degeneracy of $v$ (and $Z$). By the descriptions of \eqref{equ:general-pt}, \eqref{equ:marked-pt}, \eqref{equ:node}, and \eqref{equ:sec-pole}, the above morphism is well-defined away from the incoming nodes of $Z$ (those joining $Z$ to bigger components under $\geq_{\min}$), hence it extends to a meromorphic morphism 
\begin{equation*}
\eta^{\beta}_{Z}: \cO_{\uL}|_{Z} \to \cO_{Z} 
\end{equation*}
with poles of contact order $c_{z}$ at each incoming node $z$ of $Z$. Therefore, $\eta^{\beta}_{Z}$ corresponds to a meromorphic differential on $Z$ such that
\begin{enumerate}
 \item $\eta^{\beta}_{Z}$ has a zero of order $m_i$ at each marking $\sigma_{i} \in Z$;
 \item $\eta^{\beta}_{Z}$ has a zero of order $c_{z} - 1$ at a node $z \in Z$ if the corresponding edge in $G$ is oriented from the vertex of $Z$ to another vertex;  
 \item $\eta^{\beta}_{Z}$ has a simple pole at an unoriented node $z \in Z$ where $c_{z} = 0$;   
 \item $\eta^{\beta}_{Z}$ has a pole of order $c_{z} + 1$ at a node $z \in Z$ if the corresponding edge in $G$ is oriented from another vertex to $Z$.
\end{enumerate}

Note that the zero or pole orders of the differential $\eta^{\beta}_{Z}$ at the node $z$ differ from the contact order $c_z$ by one. This is due to the fact that the dualizing line bundle is locally generated by differentials with simple poles at the nodes. 

\begin{definition}
In the above setting, we say that $\eta^{\beta}_{Z}$ is the {\em meromorphic differential induced by the chart $\beta$}, or simply the {\em induced differential} if there is no confusion. 
\end{definition}

Next we study how induced differentials behave with different charts.

\begin{proposition}\label{prop:diff-comp}
Consider two charts $\beta_i : \ocM_{S} \to \cM_{S}$ for $i=1,2$. 
Let $\eta^{\beta_{i}}_{Z}$ be the meromorphic differential induced by $\beta_i$. Then we have
\[
u \cdot \eta^{\beta_1}_{Z} = \eta^{\beta_2}_{Z}
\]
where $u$ is the non-zero constant satisfying  
\[
\beta_{1}(e_{Z}) = u + \beta_{2}(e_{Z}).
\]
Furthermore, the induced differential $\eta^{\beta_{i}}_{Z}$ only depends on the lifting $\beta(e_{Z}) \in \cM_{S}$ of $e_{Z} \in \ocM_{S}$.
\end{proposition}
\begin{proof}
For the first statement, note that $\beta_1$ and $\beta_2$ are both sections of the quotient $\cM_{S} \to \ocM_{S}$. Then the difference $\beta_{1}(e_{Z}) - \beta_{2}(e_{Z})$ lies in the subgroup $\cO^*_{S} \subset \cM_{S}$, because its image in $\ocM_{S}$ is trivial. Now the statement follows from \eqref{equ:diff-comp}. 

For the second statement, if $e_{Z} \neq 0$, then it follows from (\ref{equ:diff-comp}) and  the construction of $\eta^{\beta_{i}}_{Z}$. If $e_{Z} = 0$, then any lift $e \in \cM_{S}$ of $e_Z$ is necessarily a non-zero constant. Note that if $e \neq 1$, then there does not exist any chart $\beta$ with $e = \beta(e_{Z})$. However, one can still carry out the construction \eqref{equ:diff-comp} in this case and obtain a differential $\eta^{e}_{Z}$ on $Z$. We observe that 
\[
\eta^{e}_{Z} = e \cdot \eta|_{Z}.
\]
This finishes the proof.
\end{proof}

By the above proposition, for any information about $\eta_{Z}^{\beta}$ that is independent of scaling, we can drop the subscript $\beta$. In particular, we have the following definition. 

\begin{definition}\label{def:pole-zero}
In the above setting, a point $z \in Z$ is called a {\em pole} (or {\em zero}) of $\eta_{Z} := \eta|_{Z}$,  if it is a pole (or zero) of $\eta^{\beta}_{Z}$ for one (hence every) chart $\beta: \ocM_{S} \to \cM_{S}$. We denote by 
\[\ord_{z}\eta_{Z} = \ord_{z}\eta^{\beta}_{Z}\] 
the (zero or pole) {\em order} of the induced differential $\eta^{\beta}_{Z}$ at $z$ for some chart $\beta$. 
\end{definition}

By Proposition \ref{prop:diff-comp}, $\ord_{z}\eta_{Z}$ is independent of the choice of charts. If $z$ is a pole, then it is an incoming node of $Z$ with pole order $c_z + 1$ for the induced differential $\eta_{Z}$.

\section{Spin structure}\label{sec:spin}

In this section we show that the moduli space of log twisted differentials distinguishes the spin parity in the boundary. 

\subsection{Statements and results}\label{ss:spin-statement}
Consider the case $\mu = (m_{i})_{i=1}^{n}$ with each $m_{i}$ even. Let $\eta$ be a non-degenerate differential  with signature $\mu$ on a smooth algebraic curve $C$ over a geometric point $S$. In this case, the log structure $\cM_S$ is trivial, and $\eta$ defines the  canonical divisor 
\[
(\eta) := \sum_{i=1}^{n} m_{i}\cdot \sigma_{i}.
\]
Since each $m_{i}$ is even, we define the {\em spin divisor} 
\[
(\sqreta) := \sum_{i=1}^{n} \frac{m_{i}}{2}\cdot\sigma_{i}.
\]
We say that $\eta$ carries an {\em even} or {\em odd spin structure} if 
\[
\dim H^0(\cO_{C}(\sqreta))
\]
is even or odd respectively. In this section, we will extend the notion of spin parity to the degenerate case, and show that the even and odd spin structures do not mix up in the boundary of the log compactification. More precisely, we prove the following. 

\begin{thm}\label{thm:spin}
Consider the partition $\mu = (m_{i})_{i=1}^{n}$ with each $m_{i}$ even. Then we have a disjoint union 
\[
\lH(\mu) = \lH(\mu)^{+}\sqcup \lH(\mu)^{-},
\]
where $\lH(\mu)^{+}$ (resp. $\lH(\mu)^{-}$) is the open and closed substack of $\lH(\mu) $ parameterizing log twisted differentials with even (resp. odd) spin parity.
\end{thm}

Spin parity for log twisted differentials is defined in Proposition-Definition \ref{prop-def:spin}. In general, spin bundles only exist over orbifold covers of the underlying stable curves. However, we prove that the spin parity does not depend on various choices in the construction, hence is well-defined for a given log twisted differential.  The above theorem then follows from  Proposition~\ref{prop:spin-inv} which states that the spin parity remains constant on each connected family of log twisted differentials. 

\subsection{An easy case of spin parity on the boundary}
In this section, we first discuss how to define the spin structure of the boundary objects under certain simplifications. 

Let $\eta$ be a log twisted differential over a log curve $C \to S$ with $\uS$ a geometric point. Let $\mu = (m_i)_{i=1}^{n}$ be the signature of $\eta$ with each $m_i$ even. We make the following two assumptions on $\eta$:

\begin{assumption}\label{assu:spin}
Let $G$ be the graph associated to $\eta$. We assume that
\begin{enumerate}
\item For each edge $l \in E(G)$, the contact order $c_l \in \NN$ is even;
\item For each vertex $v \in V(G)$, the degeneracy $e_v$ is divisible by $2$, i.e., $\frac{1}{2}\cdot e_v \in \ocM(G)$.
\end{enumerate}
\end{assumption}

Note that the log twisted differential $\eta: C \to L$ is determined by $\eta': C \to L'$ where $L' = (\uL, \cN'_{L})$ as in Section \ref{ss:log-twist-diff}. Since $\cN'_{L}$ is defined by the zero section of the line bundle $\uL \to C$, it is a Deligne-Falitings log structure of rank one, see \cite[Complement 1]{KKato} or \cite[Appendix A.2.]{Chen}. In particular, we have a morphism 
\begin{equation}\label{equ:DF}
\bar{\theta}: \NN \to \ocN'_{L}
\end{equation}
which locally lifts to a chart of $\cN'_{L}$. Here $\NN$ is viewed as the constant sheaf on $\uC$. We form the following commutative diagram
\begin{equation*}
\xymatrix{
                                       & \cN'_{L} \ar[r]^{(\eta')^{\flat}} \ar[d] & \cM_{C} \ar[d]_{\frq} \\
 \NN \ar[r]^{\bar{\theta}}     & \ocN'_{L} \ar[r]^{(\bar{\eta}')^{\flat}} & \ocM_{C}. 
}
\end{equation*}
Denote by $\delta \in \ocM_{C}$ the image of $1 \in \NN$ via the bottom composition. The assumption (\ref{assu:spin}) implies that $\frac{\delta}{2} \in \ocM_{C}$ is a global section as well. Denote by $\cT = \frq^{-1}(\frac{\delta}{2})$. We observe that $\cT$ is an $\cO^*_{\uC}$-torsor over $\uC$. Indeed, let $E$ be the line bundle on $\uC$ associated to $\cT$. Since $\uL$ is the dualizing line bundle over $\uC$, we have an isomorphism of line bundles
\[
E^{\otimes 2} \cong \uL^{\vee}.
\]
We thus obtain a {\em spin bundle} $E^{\vee}$ over $\uC$ associated to $\eta$. 

\begin{definition}
Under Assumption~(\ref{assu:spin}), we say that $\eta$ carries an even (resp. odd) spin structure if $\dim H^0(E^{\vee})$ is even (resp. odd). 
\end{definition}

\begin{proposition}
The spin bundle associated to $\eta$ is the same as the spin bundle associated to the corresponding minimal object $\eta_{min}$.
\end{proposition}
\begin{proof}
Let $\eta'_{min}: C_{min} \to L'$ be the corresponding morphism of log schemes of $\eta_{min}$. We have the commutative diagram
\begin{equation*}
\xymatrix{
                                       & \cN'_{L} \ar[r]^{(\eta'_{min})^{\flat}} \ar[d] & \cM_{C_{min}} \ar[d]_{\frq_{min}} \ar[r] & \cM_{C} \ar[d]_{\frq} \\
 \NN \ar[r]^{\bar{\theta}}     & \ocN'_{L} \ar[r]^{(\bar{\eta}'_{min})^{\flat}} & \ocM_{C_{min}} \ar[r] & \ocM_{C}. 
}
\end{equation*}
Denote by $\delta_{min} \in \ocM_{C_{min}}$ the image of $1 \in \NN$ of the first two bottom maps. Clearly, the image of $\delta_{min}$ in $\ocM_{C}$ is $\delta$.  Denote by $\cT_{\min} = \frq_{min}^{-1}(\frac{\delta_{min}}{2})$. Then the above commutative diagram induces an isomorphism of the torsors
\[
\cT_{min} \to \cT.
\]
This proves the statement. 
\end{proof}

In general, if the two assumptions in (\ref{assu:spin}) do not hold, extra modifications are necessary for defining the spin parity. First, we use the orbifold approach of \cite{AJ} to allow orbifold structures along the nodes with odd contact orders. The auxilary orbifold structures guarantee that all contact orders in the orbifold case become even. Then we will need to apply further base change of the log structure of the base to make sure that all degeneracies are divisible by $2$. Finally, we verify that the spin parity is independent of various choices in the above modifications.

\subsection{Orbifolding along odd contact nodes}\label{ss:orbi-curve}

Below we introduce the auxiliary orbifold curve. First consider the case of a geometric fiber. Let $\eta$ be a minimal log twisted differential over a log curve $C \to S$ with $\uS$ a geometric point. Let $G$ be the weighted graph of $\eta$. Consider the canonical isomorphism of monoids $\psi: \ocM(G) \to \ocM_{S}$ as in (\ref{equ:minimal}). For simplicity, we will identify $e \in \ocM(G)$ with its image $\psi(e) \in \ocM_{S}$ when there is no confusion. 

We take an orbifold nodal curve $\tilde{\uC}$ over $\uS$ as in \cite[Definition 4.1.2]{AV}\footnote{Such $\tilde{\uC}$ is called a twisted nodal curve in \cite{AV}. Here we use ``orbifold'' to avoid possible confusion, as ``twist'' has been already used for log twisted differentials.} with the following properties: 
\begin{enumerate}
 \item The coarse moduli morphism is $\underline{\frt}_{c}: \tilde{\uC} \to \uC$; 
 \item The stacky locus of $\tilde{\uC}$ consists of the nodes whose images in $\uC$ are nodes with odd contact orders;
 \item Each stacky node of $\tilde{\uC}$ is a $\mu_{2}$-gerb over $\uS$, where $\mu_2$ is the multiplicative group of order $2$.
\end{enumerate}
Note that the morphism $\underline{\frt}_{c}: \tilde{\uC} \to \uC$ is isomorphic away from the stacky nodes, and the orbifold curve $\tilde{\uC}$ is uniquely determined by the underlying curve $\uC$ (see \cite[Section 1.2]{AF}). Indeed, let $\tz \in \tilde{\uC}$ be a stacky node with its image $z \in \uC$. Then \'etale locally around $\tz$, the orbifold curve $\tilde{\uC}$ is given by the stack quotient
\begin{equation*}
\left[ (\spec k[\tx, \ty]/(\tx\cdot\ty)) \Large/ \mu_{2} \right]
\end{equation*}
where the group action of $\mu_{2}=\{1 = \zeta^2, \zeta\}$ is given by 
\[\zeta\cdot (\tx, \ty) = (\zeta\cdot \tx, \zeta\cdot \ty).\] 
Assume that locally near $z$, the underlying curve $\uC$ is of the form 
\[
\spec k[x, y]/(x\cdot y).
\]
The morphism $\underline{\frt}_{c}$ locally near $\tz$ has the form 
\begin{equation}\label{equ:odd-node-map}
\underline{\frt}^*_{c}: (x, y) \mapsto (\tx^2, \ty^2).
\end{equation}

\subsection{Canonical log structure on orbifold curves} \label{ss:orbi-log-curve}

By \cite[Theorem 3.6]{LogCurve}, there is a canonical log structure on the orbifold curve $\tilde{\uC} \to \uS$, which yields a family of log orbifold curves:
\[
\tpi^{\sharp} : \tC^{\sharp} \to \tS^{\sharp}.
\]
By \cite[Corollary 4.7]{LogCurve}, we have a commutative diagram
\begin{equation}\label{equ:orbi-log}
\xymatrix{
\tC^{\sharp} \ar[r]^{\frt_{c}} \ar[d] & C^{\sharp} \ar[d] \\
\tS^{\sharp} \ar[r]^{\frt_{b}} & S^{\sharp}.
}
\end{equation}
where the underlying morphism of $\frt_b$ is an isomorphism of the underlying scheme $\uS$. 

Note that both $\uC$ and $\tilde{\uC}$ have the same dual graph $\uG$. Let $k = |E(G)|$ be the total number of edges, $k_{1}$ be the number of edges with odd contact orders, and $k_{2}$ the number of edges with even contact orders, so $k = k_1 + k_2$. Next, we will use notations in Section \ref{ss:log-curve} for the log curve $\pi^{\sharp}: C^{\sharp} \to S^{\sharp}$. Following \cite[Section 4]{LogCurve}, we give a description of (\ref{equ:orbi-log}) on the level of characteristic sheaves. \\

\noindent {\bf The base.} Over $\uS$, we have
\[
\ocM_{\tS^{\sharp}} = \NN^{k} = \NN^{k_{1}}\oplus \NN^{k_{2}}
\] 
with the generators one-to-one corresponding to the edges of $\uG$. For each edge $l \in E(G)$, denote by $\tilde{e}_{l}$ the corresponding generator of $\ocM_{\tS^{\sharp}}$. We call $\tilde{e}_{l}$ the {\em smoothing parameter} of the node corresponding to $l$. The morphism 
\begin{equation*}
\bar{\frt}^{\flat}_{b}: \ocM_{S^{\sharp}} \to \ocM_{\tS^{\sharp}}
\end{equation*}
is defined by $e_{l} \mapsto \tilde{e}_{l}$ if $l$ has even contact order, and defined by $e_{l} \mapsto 2\cdot \tilde{e}_{l}$ if $l$ has odd contact order.

\vspace{.5mm}

\noindent
{\bf Smooth unmarked points of $\tilde{\uC}$. }
At a smooth unmarked point $z \in \tilde{\uC}$, we have $\ocM_{\tC^{\sharp},z} = \ocM_{\tS^{\sharp}}$, and the morphism of monoids
\begin{equation}\label{equ:orbi-smooth-pt}
\bar{\frt}_{c,z}^{\flat} = \bar{\frt}^{\flat}_{b}: \ocM_{C^{\sharp}, z} = \ocM_{S^{\sharp}} \to \ocM_{\tC^{\sharp},z}  = \ocM_{\tS^{\sharp}}. 
\end{equation}

\noindent
{\bf Marked points.}
For a marked point $\sigma_{i} \in \uC$, denote by $\tilde{\sigma}_{i} \in \tilde{\uC}$ the corresponding marked point. Then we have $\ocM_{\tC^{\sharp}, \tilde{\sigma}_{i}} = \ocM_{\tS^{\sharp}} \oplus \NN$ with the copy of $\NN$ corresponding to the divisorial log structure on $\tilde{\uC}$ given by $\tilde{\sigma}_{i}$. We then have
\begin{equation}\label{equ:orbi-mark-pt}
\bar{\frt}^{\flat}_{c} = \bar{\frt}^{\flat}_{b}\oplus\id_{\NN}: \ocM_{C^{\sharp}, \sigma_{i}} = \ocM_{S^{\sharp}}\oplus\NN \to \ocM_{\tC^{\sharp}, \tilde{\sigma}_{i}} = \ocM_{\tS^{\sharp}} \oplus \NN.
\end{equation}

\noindent
{\bf Nodes.}
Consider a node $z \in \tilde{\uC}$ corresponding to an edge $l$. Denote by $\tx$ and $\ty$ the local coordinate functions on the two components of $\tilde{\uC}$ meeting at $z$. Then we have the canonical local lifts $\tlx$ and $\tly$ in $\cM_{\tC^{\sharp},z}$. On the characteristic level, we have 
\[
\ocM_{\tC^{\sharp},z} = \ocM_{\tS^{\sharp}}\oplus_{\NN}\NN^{2}
\]
with the morphism $\NN \to \NN^{2}$ on the right given by $1 \mapsto \lx + \ly$, and $\NN \to \ocM_{\tS^{\sharp}}$ given by $1 \mapsto \tilde{e}_{l}$. Over the node $z$, consider the morphism of monoids
\[
\bar{\frt}^{\flat}_{c, z}: \ocM_{S^{\sharp}}\oplus_{\NN}\NN^{2} \to \ocM_{\tS^{\sharp}}\oplus_{\NN}\NN^{2}.
\]
If the edge $l$ has an even contact order, then we have 
\begin{equation}\label{equ:even-node-char}
\bar{\frt}^{\flat}_{c, z} = \bar{\frt}_{b}\oplus \id_{\NN^2}
\end{equation} 
which induces an isomorphism on the $\NN^2$ summand. 
If the edge $l$ has an odd contact order, then we have 
\begin{equation}\label{equ:odd-node-char}
\bar{\frt}^{\flat}_{c, z} = \bar{\frt}_{b}\oplus \bar{\nu}_{2}
\end{equation}
with the morphism $\bar{\nu}_{2}: \NN^2 \to \NN^2$ given by $(1,1) \mapsto (2,2)$. Note that such a morphism on the level of log structures is compatible with the underlying structure in (\ref{equ:odd-node-map}). 

\subsection{The orbifold lift of $\eta$} \label{ss:orbifold-lift}

Consider the following cartesian diagram of log schemes
\begin{equation}\label{equ:changeto-orbi-base}
\xymatrix{
\tS \ar[r] \ar[d] & S \ar[d] \\
\tS^{\sharp} \ar[r] & S^{\sharp}.
}
\end{equation}

\begin{lemma}\label{lem:log-produce-surj}
The underlying morphism of $\tS \to \tS^{\sharp}$ is finite and surjective. In particular, $\tS \to S$ is finite and surjective. 
\end{lemma}
\begin{proof}
Note that the morphism $\tS^{\sharp} \to S^{\sharp}$ is integral, see \cite[Definition (4.3)]{KKato}. Hence by \cite[Proposition 2.4.2]{Ogus}, the fiber product $S\times_{S^{\sharp}}\tS^{\sharp}$ in the category of fine log schemes is non-empty. The statement then follows from \cite[Proposition 2.4.5 (2)]{Ogus}. 
\end{proof}

Since $\uS$ is a geometric point, the underlying scheme $\tilde{\uS}$ is possibly a disjoint union of several copies of $\uS$. It is useful to notice that for each component $\tS_{i} \subset \tS$, the characteristic monoid 
\[
\ocM_{\tS_{i}} = \ocM_{\tS^{\sharp}}\oplus_{\ocM_{S^{\sharp}}}\ocM_{S}
\]
in the category of fine and saturated sharp monoids is identical to each other by the cartesian diagram \eqref{equ:changeto-orbi-base}. We also observe that 

\begin{lemma}\label{lem:first-base-change}
The morphism $\ocM_{S}^{gp} \to \ocM_{\tS_{i}}^{gp}$ is injective with torsion cokernel.
\end{lemma}

Now consider the fiber product of log stacks
\begin{equation*}
\xymatrix{
\tC \ar[r] \ar[d]^{\tpi} & \tC^{\sharp} \ar[d]^{\tpi^{\sharp}} \\
\tS \ar[r] & \tS^{\sharp}.
}
\end{equation*}

Note that the underlying stack of $\tC$ is the orbifold curve $\tilde{\uC}$ described in Section \ref{ss:orbi-curve}. Consider the corresponding dualizing line bundles $\omega_{\tilde{C}}$ and $\omega_{C}$ over $\tilde{\uC}$ and $\uC$ respectively. Let $\cN'_{\omega_{\tilde{C}}}$ and $\cN'_{\omega_{C}}$ be the divisorial log structures on the total spaces of $\tilde{\uC}$ and $\uC$ associated to the zero sections respectively. Denote by $\omega_{\tilde{C}}' = (\omega_{\tilde{C}}, \cN'_{\omega_{\tilde{C}}})$ and $\omega_{C}' = (\omega_{C}, \cN'_{\omega_{C}})$ the corresponding log schemes.  

We also consider $\cN_{\omega_{\tilde{C}}}$ and $\cN_{\omega_{C}}$ the log structures on $\omega_{\tilde{C}}$ and $\omega_{C}$ obtained by pulling back the log structures of $\tC$ and $C$ respectively. Denote by $\tilde{L} = (\omega_{\tilde{C}}, \cN_{\omega_{\tilde{C}}}\oplus_{\cO^*}\cN'_{\omega_{\tilde{C}}})$ and $L = (\omega_{C}, \cN_{\omega_{C}}\oplus_{\cO^*}\cN'_{\omega_{C}})$. 

Since $\frt_c^*\omega_{C} = \omega_{\tilde{C}}$, the morphism 
\begin{equation*}
\omega_{\tilde{C}}' \to  \omega_{C}' 
\end{equation*}
is strict.  Now consider the composition
\[
\tC \stackrel{}{\longrightarrow} C \stackrel{\eta'}{\longrightarrow} L'.
\]
It fits in a commutative diagram of solid arrows
\begin{equation*}
\xymatrix{
\tC \ar@/^/[rrd] \ar@{-->}[rd]^{\xi'} \ar@/_/[ddr]&& \\
 & \tilde{L}' \ar[r] \ar[d] &  L'  \ar[d] \\
 & \underline{\tC} \ar[r] & \uC.
}
\end{equation*}
where the square is cartesian. We thus obtain a morphism 
\[
\xi': \tC \to \tL'
\]
making the above diagram commutative. Define 
\[
\xi: \tC \to \tL
\]
whose underlying structure is the same as that of $\xi'$, and on the level of log structures
\[
\xi^{\flat} := \id_{\cN_{\tL}}\oplus(\xi')^{\flat}: \cN_{\tL}\oplus_{\cO^*}\cN'_{\tL} \to \cN_{\tL}.
\]
Then we have the following commutative diagram with a cartesian square given by the solid arrows
\begin{equation}\label{equ:orbi-pull-back}
\xymatrix{
\tL \ar[r] \ar[d] & L \ar[d] \\
\tC \ar[r] \ar@/^1pc/@{-->}[u]^{\xi} & C  \ar@/_1pc/@{-->}[u]_{\eta}. 
}
\end{equation}
We make the following observation. 

\begin{lemma}\label{lem:orbi-base-change}
The morphism $\xi$ is the pull-back of $\eta$, and it makes the diagram \eqref{equ:orbi-pull-back} commutative. 
\end{lemma}

We call $\xi$ the {\em orbifold lift} of $\eta$.

\subsection{Spin bundles of the orbifold lifts}\label{ss:spin-bundle}

Note that $\tS$ is possibly a disjoint union of several log points. We consider a further base change:

\begin{equation}\label{equ:auxiliary-base-change}
\tS_{1} \to \tS
\end{equation}
with the properties that 
\begin{enumerate}
 \item $\tilde{\uS}_{1} = \uS$, and the underlying morphism of $\tS_{1} \to \tS$ is an inclusion of point;  
 \item Let $e'_{v}$ be the image of the degeneracy $e_{v}$ of each vertex $v \in V(G)$ via the composition 
 \[
 \ocM(G) \to \ocM_{S} \to \ocM_{\tS} \to \ocM_{\tS_{1}}.
 \]
 Then $\frac{1}{2}\cdot e'_{v} \in \ocM_{\tS_{1}}$. 
\end{enumerate}

Such a base change exists as charts exist over a geometric point, and a refinement of the lattice structure will yield (2). Note that the second property is to address Assumption \ref{assu:spin} (2) in the general situation.  

Same as in (\ref{equ:DF}), we have a morphism of sheaves of monoids
\[
\bar{\theta}: \NN \to \ocN'_{\tL}
\]
which lifts to a chart \'etale locally. Here we view $\NN$ as the global constant sheaf of monoids with coefficients in $\NN$. Consider the following commutative diagram
\begin{equation}\label{diag:spin-torsor}
\xymatrix{
                                       & \cN'_{\tL} \ar[r]^{(\xi')^{\flat}} \ar[d] & \cM_{\tilde{C}} \ar[d]_{\frq} \\
 \NN \ar[r]^{\bar{\theta}}     & \ocN'_{\tL} \ar[r]^{(\bar{\xi}')^{\flat}} & \ocM_{\tilde{C}}. 
}
\end{equation}
Consider the element $\bar{\delta} = (\bar{\xi}')^{\flat}\circ\bar{\theta}(1)$. We observe the following. 

\begin{lemma}\label{lem:divide-section}
In the above setting, we have $\frac{1}{2}\cdot\bar{\delta} \in \ocM_{\tilde{C}}$.
\end{lemma}
\begin{proof}
Note that $\xi'$ is the pull-back of $\eta'$. By the second property of the base change (\ref{equ:auxiliary-base-change}), away from the orbifold nodes, the element $\frac{1}{2}\cdot\bar{\delta}$ lies in $\ocM_{\tilde{C}}$.

Now consider an orbifold node $z\in \tC$ corresponding to an edge $l \in V(G)$ joining two vertices $v_{1}$ and $v_{2}$ with $v_{1} \leq_{\min} v_{2}$ and contact order $c_l$. By Lemma \ref{lem:orbi-base-change}, \eqref{equ:odd-node-char}, and \eqref{equ:orbi-pull-back}, we have 
\begin{equation*}
(\bar{\xi}')^{\flat}_{z}: \NN \to \ocM_{\tS}\oplus_{\NN}\NN^2, \ \ \ 1 \mapsto e_{v_{1}} + 2\cdot c_l\cdot \lx,
\end{equation*} 
where $e_{v_1}$ is the image of the degeneracy. Since $\frac{e_{v_1}}{2} \in \ocM_{\tC}$, we have $\frac{1}{2}\cdot\bar{\delta} \in \ocM_{\tilde{C}}$ at the orbifold nodes as well.
\end{proof}

Consider the inverse image
\begin{equation}\label{equ:torsor-root}
\cT = \frq^{-1}(\frac{1}{2}\cdot\bar{\delta})\in  \cM_{\tilde{C}}. 
\end{equation}
It is an $\cO^*$-torsor over $\tC$. Denote again by $E$ the line bundle corresponding to $\cT$. By construction, we have
\begin{equation}\label{equ:spin-bundle}
E^{\otimes 2} \cong \omega_{\tilde{C}}^{\vee}.
\end{equation}
We call $E^{\vee}$ the {\em spin bundle} over $\underline{\tC}$ associated to $\xi$. The spin parity of $\xi$ is even (resp. odd) if $H^0(E^{\vee})$ is even (resp. odd). 

\begin{remark}
The curve $\tilde{\uC}$ with the spin bundle $E^{\vee}$ is called a twisted $2$-spin curve in \cite[Section 1.4]{AJ}.
\end{remark}

\subsection{Independence of the spin parity}
\begin{proposition-definition}\label{prop-def:spin}
Notations as in Section \ref{ss:spin-bundle}, the spin parity of $\xi$ only depends on the minimal object $\eta_{min}$ associated to $\eta$. We thus define the spin parity of $\eta$ to be the spin parity of $\xi$.
\end{proposition-definition}

The proof of the proposition will occupy this section. For simplicity, we may assume that $\eta$ is minimal. We first observe the following.  

\begin{lemma}
The spin bundle $E^{\vee}$ is representable.
\end{lemma}
\begin{proof}
Consider an orbifold node $z \in \tC$ corresponding to an edge $l$. The isotropy group $\mu_{2}$ acts on $\cT$ faithfully as the contact order $c_l$ is odd.
\end{proof}

The above lemma implies that global sections of $E^{\vee}$ vanish along the orbifold nodes of $\underline{\tC}$. Denote by $\cup_{j} \uZ_{j}$
the union of connected components of $\underline{\tC}$, obtained by taking partial normalizations along the orbifold nodes. 

Now consider two different base changes $\tS_{i} \to \tS$ as in \eqref{equ:auxiliary-base-change} for $i=1,2$. Denote by $\xi_{i}: \tC_{i} \to \tL_{i}$ the corresponding orbifold lift of $\eta$ over $\tS_{i}$ for $i=1,2$. Let $\cT_{i}$ be the torsor associated to $\xi_i$ as in \eqref{equ:torsor-root}. To prove Proposition-Definition \ref{prop-def:spin}, we next construct an explicit isomorphism of torsors
\begin{equation}\label{equ:compare-spin}
\cT_{1}|_{\uZ_j} \to \cT_{2}|_{\uZ_j}
\end{equation}
for each connected components $\uZ_j$. 

Choose a chart $\psi: \ocM(G) \to \cM_{S}$, and consider the composition $\psi_i$:
\[
\ocM(G) \to \cM_{S} \to \cM_{\tS_i}.
\]
Fix $j$ and a vertex $v \in V(G)$ such that the corresponding irreducible component $\uZ_v$ of $\underline{\tC}$ is contained in $\uZ_j$. Fix a choice of $\theta_{i,v} \in \cM_{\tS_i}$ such that 
\begin{equation*}
2 \cdot \theta_{i,v} = \psi_i(e_v)
\end{equation*}
where $e_v \in \ocM(G)$ is the degeneracy of $v$. 

Now let $G_j$ be a connected subgraph of $G$ corresponding to $\uZ_j$. For another vertex $v' \in V(G_j)$ corresponding to an irreducible component $\uZ_{v'} \subset \uZ_j$, we choose an oriented path 
\[
\vec{P} = (l_1, \cdots, l_k)
\]
in $G_j$ joining $v$ and $v'$. Note that in $\ocM(G)^{gp}$ we have the relation
\begin{equation*}
e_{v'} = e_v + \sum_{m=1}^{k} c'_{l_m}\cdot e_{l_m}
\end{equation*}
where $c'_{l_m} = c_{l_m}$ if the orientation of $\vec{P}$ is compatible with the orientation of $l_m$, and $c'_{l_m} = - c_{l_m}$ otherwise. Note that for each edge $l_m$, the contact order $c'_{l_m}$ is even. Applying $\psi$ and multiplying both sides by $\frac{1}{2}$ in the above equality, we have 
\begin{equation*}
\frac{1}{2} \cdot \psi_{i}(e_{v'})= \theta_{i,v} + \sum_{m=1}^{k} \frac{1}{2}\cdot c'_{l_m}\cdot \psi_i(e_{l_m}) \in \cM_{\tS_i}.
\end{equation*}
Set $\theta_{i,v'} = \frac{1}{2} \cdot \psi_{i}(e_{v'})$. Since the degeneracy $e_{v'}$ does not depend on the choice of $\vec{P}$, the element $\theta_{i,v'}$ does not depend on the choice of $\vec{P}$ either.

\begin{lemma}\label{lem:componentwise}
Notations and assumptions as above, for each $v' \in G_j$, the correspondence 
\begin{equation}\label{equ:componentwise}
\theta_{1,v'} + u \mapsto \theta_{2,v'} + u
\end{equation}
over the smooth unmarked locus of $\uZ_{v'}$ for $u \in \cO^*_{\uZ_v}$ induces an isomorphism of torsors $\cT_{1}|_{\uZ_{v'}} \to \cT_{2}|{\uZ_{v'}}$.  
\end{lemma}
\begin{proof}
Note that over the smooth unmarked locus of $\uZ_{v'}$, the image of $\theta_{1,v'}$ in $\ocM_{\tC_i}$ is $\frac{1}{2}\cdot \bar{\delta}$, see (\ref{diag:spin-torsor}) for the definition of $\bar{\delta}$. The above isomorphism over the smooth unmarked locus extends to $\uZ_{v'}$ by the local description of the orbifold curves as in (\ref{equ:orbi-smooth-pt}), (\ref{equ:orbi-mark-pt}), (\ref{equ:even-node-char}), and (\ref{equ:odd-node-char}).
\end{proof}

\begin{lemma}
The isomorphisms $\cT_{1}|_{\uZ_{v'}} \to \cT_{2}|_{\uZ_{v'}}$ defined in Lemma \ref{lem:componentwise} over each irreducible component can be glued together, and yield an isomorphism (\ref{equ:compare-spin}).
\end{lemma}
\begin{proof}
It suffices to show that the isomorphisms $\cT_{1}|_{\uZ_{v'}} \to \cT_{2}|_{\uZ_{v'}}$ are compatible at the nodes of $\uZ_{j}$. Let $l$ be an edge corresponding to a node $z \in \uZ_j$. We fix the two local coordinates $x$ and $y$ near the node of the two components $Z_{v_1}$ and $Z_{v_2}$ joining at $z$ respectively. Assume that $v_2\geq_{\min}v_1$ with $c_l \geq 0$, see Section \ref{ss:minimality} for the definition of $\geq_{\min}$. Then by the construction in (\ref{equ:componentwise}), locally near $z$, we obtain an isomorphism of the two different torsors induced by 
\begin{equation}\label{equ:compatible-glue1}
\theta_{1,v_1} + \frac{1}{2}c_l \lx + u \mapsto \theta_{2,v_1} + \frac{1}{2}c_l \lx + u
\end{equation}
over $Z_{v_1}$, and an isomorphism induced by 
\begin{equation}\label{equ:compatible-glue2}
\theta_{1,v_1} + \frac{1}{2}c_l e_l - \frac{1}{2}c_l \ly + u \mapsto \theta_{2,v_1} + \frac{1}{2}c_l e_l - \frac{1}{2}c_l \ly + u
\end{equation}
over $Z_{v_2}$. Here $u$ is an invertible function around $z$, and $\lx$ and $\ly$ are defined as in (\ref{equ:node-smoothing}). Finally, the relation (\ref{equ:node-smoothing}) implies that (\ref{equ:compatible-glue1}) and (\ref{equ:compatible-glue2}) are indeed compatible. Therefore, we obtain the isomorphism (\ref{equ:compare-spin}) as needed.
\end{proof}

This finishes the proof of Proposition-Definition \ref{prop-def:spin}.

\subsection{Deformation invariance of the spin parity}\label{ss:spin-def-inv}

Next we show that different spin parities do not mix up in a connected family. Recall that we have a partition 
$\mu = (m_i)_{i=1}^n$ of $2g-2$ with each $m_i$ even. 

\begin{lemma}\label{lem:family-orbi-lift}
For each geometric point $s \in \lH(\mu)$, there is a connected \'etale neighborhood $U \to \lH(\mu)$ of the point $s$ with the universal family $\eta: C \to L$, and a surjective morphism of log schemes $f: \tU \to U$ together with a family $\tL \to \tC \to \tU$ and a morphism $\xi: \tC \to \tL$ such that 
\begin{enumerate}
 \item For each connected component $\tU' \subset \tU$, $s\in f(\tU')$; 
 \item The underlying family $\tilde{\uC} \to \underline{\tU}$ is a family of orbifold curves whose stacky loci are given by $\mu_2$-gerbes along the nodes of odd contact orders;
 \item The morphism $\xi: \tC \to \tL$ is the pull-back of $\eta$ as in \eqref{equ:orbi-pull-back}; 
 \item At each geometric point of $\tU$, the second condition of \eqref{equ:auxiliary-base-change} is satisfied.
\end{enumerate}
\end{lemma}
\begin{proof}
Denote by $C^{\sharp} \to U^{\sharp}$ the canonical log curve associated to the underlying curve $\uC \to \uU$. By \cite[Theorem 1.9]{LogCurve}, there is a surjective morphism of connected log schemes $\tU^{\sharp} \to U^{\sharp}$ together with a commutative diagram
\[
\xymatrix{
\tC^{\sharp} \ar[r] \ar[d] & C^{\sharp} \ar[d] \\
\tU^{\sharp} \ar[r]^{f_1} & U^{\sharp}
}
\]
such that $\tC^{\sharp} \to \tU^{\sharp}$ is a family of orbifold curves with its canonical log structure, whose stacky loci are given by $\mu_2$-gerbes along the nodes of odd contact orders.

Consider the cartesian diagram
\[
\xymatrix{
\tU_2 \ar[r] \ar[d] & U \ar[d]\\
\tU^{\sharp} \ar[r]^{f_1} & U^{\sharp}. 
}
\]
By Lemma \ref{lem:log-produce-surj}, the morphism $\tU_2 \to U$ is finite and surjective. Replace $\tU_2$ by the union of its connected components whose images in $U$ contain $s$. The finiteness implies that $\tU_2 \to U$ is still a surjective morphism. 

Further shrinking $\tU_2$ (hence $U$ as well), we can also assume that a chart of $\cM_{\tU_2}$ exists. We can thus replace $\tU_2$ by a base change $\tU \to \tU_2$ whose geometric fiber satisfies the second condition of \eqref{equ:auxiliary-base-change}. This finishes the proof. 
\end{proof}

\begin{proposition}\label{prop:spin-inv}
The spin parity remains constant on every connected family of log twisted differentials.
\end{proposition}
\begin{proof}
It suffices to verify the statement for each connected component of $\lH(\mu)$. For each geometric point $s \in \lH(\mu)$, take a family $\tL \to \tC \to \tU$ and a morphism $\xi: \tC \to \tL$ as in Lemma \ref{lem:family-orbi-lift} with a surjective morphism $\tU \to U$. The spin bundle $E^{\vee}$ defined by (\ref{equ:spin-bundle}) induces a family of $2$-spin curves over the underlying $\underline{\tU}$. By \cite[Lemma 3.2]{Jarvis-Spin} and \cite[Proposition 4.3.1]{AJ}, the spin parity of the connected family over $\tU$ is determined by the spin parity over the fiber of $s$. By Proposition-Definition \ref{prop-def:spin}, the spin parity remains the same in a connected neighborhood of $s$. 
\end{proof}

\subsection{A correspondence to flat geometry}\label{ss:spin-flat}

It is known that in the Deligne-Mumford compactification, the closures of odd and even spin components may intersect (\cite{ChenDiff, Gendron}). As we have seen that adding a log structure can distinguish the spin parity on the boundary, it is natural to seek a more geometric understanding. Consequently it would give a geometric incarnation of our log structure. One such geometric candidate known to experts is via the \emph{residue slit construction} as in \cite[Section 5]{BCGGM1} (see also \cite{EMZ, BoissyFrame}). Below we briefly explain the idea. 

For a node $q\in \uC$ that glues two branches $q^+\in \uC^+$ and $q^- \in \uC^-$, suppose the underlying twisted differential $\tilde{\eta}$ has a zero of order $k$ at $q^{+}$ and a pole of order $-2-k$ at $q^{-}$. Flat geometric neighborhoods of $q^+$ and $q^-$ consist of $2k+2$ half-disks and $2k+2$ (broken) half-planes, respectively. If $\Res_{q^-}\tilde{\eta} = 0$, then one can remove the half-disks and put in the complement (broken) half-disks of the half-planes (after suitable scaling) such that $q$ can be locally smoothed out. If $\Res_{q^-}\tilde{\eta} \neq 0$, one needs to modify the neighborhood of $q^+$ to match with the residue, and the global residue condition in \cite{BCGGM1} precisely ensures that such modification can be carried out globally. Note that when we put in the (broken) half-disks by translation, there are $k+1$ choices, as there are $k+1$ horizontal lines departing from a zero of order $k$. Different choices can alter the spin parity after smoothing (\cite{EMZ, BCGGM2}). 

We illustrate the above discussion via the following example. Let $\uC$ consist of two components $\uX$ and $\uR$ meeting at two nodes $p$ and $q$, where $X$ is of genus two, $p$ and $q$ are hyperelliptic conjugate in $\uX$, and $\uR\cong \PP^1$ contains a marked point $\sigma$. Consider the twisted differential $\tilde{\eta} = (\eta_X, \eta_R)$ on $\uC$ such that $(\eta_X) = p^{+} + q^{+}$ and $(\eta_R) = 4\sigma - 3 p^{-} - 3q^{-}$. See 
Figures~\ref{fig:H(2)} and~\ref{fig:residue} for flat geometric representations of $\eta_X$ and $\eta_R$, where $p^{-}$ is at the infinity of the upper four half-disks in Figure~\ref{fig:residue}, $q^{-}$ is at the infinity of the lower four half-disks, and $\pm r$ are the residues of $\eta_R$ at $p$ and $q$. 

\begin{figure}[h]
    \centering
    \psfrag{p}{$p^+$}
    \psfrag{q}{$q^+$}
    \psfrag{l}{$l$}
    \psfrag{k}{$k$}
    \includegraphics[scale=0.7]{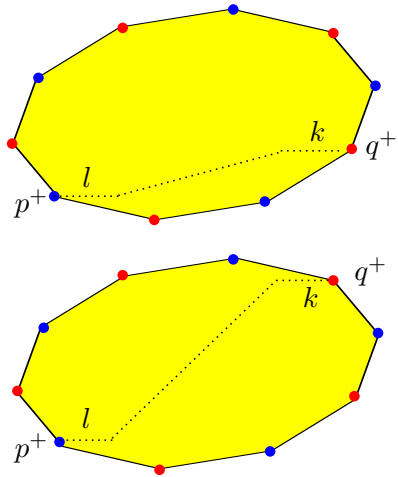}
    \caption{\label{fig:H(2)} Flat geometric representation of $\eta_X$ with two different residue slits.}
    \end{figure}

\begin{figure}[h]
    \centering
    \psfrag{a}{$l_1$}
    \psfrag{b}{$l_2$}
    \psfrag{c}{$l_{3}$}
    \psfrag{d}{$l_4$}
    \psfrag{r}{$r$}
    \psfrag{h}{$k_1$}
    \psfrag{g}{$k_{2}$}
       \psfrag{f}{$k_3$}
    \psfrag{e}{$k_{4}$}
    \includegraphics[scale=1.2]{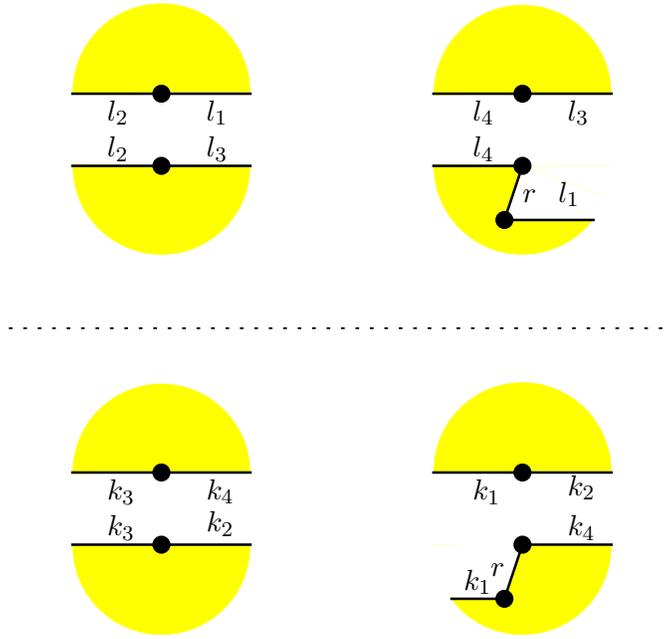}
    \caption{\label{fig:residue} Flat geometric representation of $\eta_R$.}
    \end{figure}

Now we can apply the residue slit construction in two ways to smooth $(\uC, \tilde{\eta})$ into the stratum $\cH(4)$, see the two dotted broken lines in Figure~\ref{fig:H(2)}. Along each broken line, we can remove a neighborhood of size $r$ and then perform the local modification at $p^+$ and $q^+$ as described above. Using the Arf invariant (see e.g.~\cite{KZ}), one checks that the first smoothing gives a flat surface with odd parity, and the second is with even parity. Moreover, in the latter the hyperelliptic involution is preserved by reflection from the center of the polygon. This is because in $\cH(4)$, the even spin component coincides with the hyperelliptic component. In general for higher genera, taking different residue slits not only separates the spin parity but also distinguishes the hyperelliptic structure, by breaking or preserving the hyperelliptic symmetry (see \cite{BCGGM2} for more details). 

Below we describe a heuristic correspondence between the log geometric construction and the residue slit construction. The pair of differentials $(\eta_X, \eta_R)$ can be obtained from a minimal log twisted differential $\eta$ over $C \to S$ as in Section \ref{ss:fiberwise-induced-diff} after fixing a chart $\beta: \ocM(G) \to \cM_S$ where $G$ is the weighted graph. We observe that the graph $G$ consists of two vertices $v_{R}$ and $v_{X}$ corresponding to the two components $\uR$ and $\uX$ respectively, along with two edges $l_p$ and $l_q$ both oriented from $v_{X}$ to $v_{R}$ with contact orders $c_{l_p} = c_{l_q} = 2$. We represent $G$ as follows:
\[
\xymatrix{
v_{X} \ar@/^1pc/[rr]^{l_{p}}  \ar@/_1pc/[rr]_{l_{q}} && v_{R}.
}
\]
Following Section \ref{ss:minimality}, we calculate that $\ocM(G) = \NN$. With the fixed chart $\beta: \NN \to \cM_{S}$, a similar construction as in \cite[Section 3.2]{FanoCI} combined with the discussion in \cite[6.3.1(2)]{Kim} imply that there are precisely two different choices of $\eta$. These two different choices should correspond to the two different slit constructions as above. 

Fix a choice of $\eta$, and consider two different choices of the charts $\beta_1 = \beta$ and $\beta_2$.  By Section \ref{ss:fiberwise-induced-diff}, we observe that the two induced differentials $\eta^{\beta_1}_{R} = \eta_{R}$ and $\eta^{\beta_2}$ are related by $ \eta^{\beta_2} = u \cdot \eta_{R}$, where $u \in \CC^*$ satisfies $\beta_2(1) = u \cdot \beta(1)$. This will result in a compatible rotation for the residue slit around both $p$ and $q$ in the picture shown in Figure~\ref{fig:H(2)}. Consequently, the spin parity of the nearby smooth differentials will not depend on the choice of charts. 

We emphasize again that however the definition of spin parity for log twisted differentials does not require the existence of nearby smoothing.
\section{Hyperelliptic structure}\label{sec:hyp}

In this section, we study log twisted differentials with hyperelliptic involutions. 

\subsection{Statements and results}

Recall that a hyperelliptic differential has the underlying smooth curve being hyperelliptic and the differential being anti-invariant under the hyperelliptic involution. In general the hyperelliptic involution does not lift to an involution of log twisted differentials. It is important to notice that even when it lifts, the action on a log twisted differential is not preserved under arbitrary base change in the category of log schemes. This issue can be fixed by combining the techniques of log geometry and admissible covers, which will lead to a 
well-defined fiber category of log twisted hyperelliptic differentials over the category of log schemes, denoted by $\Hyp(\mu)$ (see Section \ref{ss:hyp-moduli}). Our goal is to prove the following result. 

\begin{thm}\label{thm:hyp}
The fibered category $\Hyp(\mu)$ is represented by a separated, log smooth Deligne-Mumford stack with its universal minimal log structure. Furthermore, the forgetful morphism to the Hodge bundle $\Hyp(\mu) \to \cH_{g,n}$ is representable and finite.
\end{thm}
This Theorem follows from Propositions~\ref{prop:hyp-algebracity}, \ref{prop:hyp-prop}, and~\ref{prop:hyp-smooth}. 

\subsection{Log admissible covers}\label{ss:log-adm}

Here we first recall the notion of log admissible cover as in \cite[Section 7.2]{Kim}.  

\begin{definition}\label{def:log-adm}
A {\em minimal log admissible cover} over $\uS$ is a commutative diagram of log schemes:
\begin{equation}\label{diag:log-adm}
\xymatrix{
C_a\ar[rr] \ar[rd] && \PP_a \ar[ld] \\
 & S_a & 
}
\end{equation}
such that
\begin{enumerate}
 \item $\uPP_a \to \uS_a$ is a family of stable rational curves with two disjoint sets of markings 
\[
\cR = \{r_{i}\}_{i=1}^{n_1}, \quad \cU = \{u_j\}_{j=1}^{n_2};  
\]

\item $\PP_a \to S_a$ is a family of genus zero log curves with the canonical log structure associated to the underlying family $\uPP_a \to \uS_a$ with the set of markings $\cR \cup \cU$; 

\item $C_a \to S_a$ is a family of log curves with two disjoint sets of markings 
\[
\cR' = \{r'_{i}\}_{i=1}^{n_1}, \quad \cU' = \{u'_j, u''_j\}_{j=1}^{n_2}; 
\]

\item The underlying diagram of \eqref{diag:log-adm} is a family of hyperelliptic admissible covers with simple ramification points labeled by the set $\cR$; 

\item The morphism $C_a \to \PP_a$ sends $r'_i$ to $r_i$ and sends $u'_j, u''_j$ to  $u_j$ for any $i,j$. 
\end{enumerate}
Pull-back of a minimal log admissible cover is defined as a cartesian diagram as usual. A diagram as \eqref{diag:log-adm} is called a {\em log admissible cover} over a log scheme $S$, if it is the pull-back of a minimal log admissible cover along a morphism of log schemes $S \to S_a$. 
\end{definition}

Denote by $\cA(n_1,n_2)$ the log stack of minimal log admissible covers with its minimal log structure. Then by \cite{Kim}, the stack $\cA(n_1,n_2)$ is a log smooth algebraic stack with the locally free log structure. In particular, the underlying stack of $\cA(n_1,n_2)$ is smooth. By definition, the log stack $\cA(n_1,n_2)$ represents the category of log admissible covers fibered over the category of log schemes. 

For a log admissible cover $C \to \PP$ over $S$, the hyperelliptic involution $\iota: C \to C$ is an isomorphism of log schemes over $S$ that fits in the following commutative diagram
\[
\xymatrix{
C \ar[rr]^{\iota} \ar[rd] && C \ar[ld] \\
 &\PP_{S}.&
}
\] 
It should be emphasized that $\iota$ acts trivially on $S$. Thus, the hyperelliptic involution is compatible with arbitrary base changes in the category of log schemes.

Note that for any underlying admissible cover over a geometric point, there is a unique lift to a minimal log admissible cover, see \cite[Remark 6.3.1(2)]{Kim}.

\subsection{Log twisted hyperelliptic differentials}\label{ss:hyp-diff}

\begin{definition}\label{def:hyp-diff}
A {\em log twisted hyperelliptic differential} over an arbitrary log scheme $S$ consists of the following data: 
\[
(\eta, C \to \PP_S)
\]
where
\begin{enumerate}
 \item $C \to \PP_S$ is a family of log admissible covers over $S$, see Definition \ref{def:log-adm};
 \item $\eta$ is a log twisted differential on $C \to S$. 
\end{enumerate}
They are compatible with the hyperelliptic involution $\iota$ such that $\iota^* \eta = -\eta$. 
\end{definition}

Pull-back of a log twisted hyperelliptic differential along a morphism of log schemes is defined in the usual sense, since the hyperelliptic involution is compatible with arbitrary base changes in the category of log schemes. Thus, log twisted hyperelliptic differentials form a category fibered over the category of log schemes. 


\subsection{Minimality}\label{ss:hyp-minimality}

We next turn to study minimal objects in this fibered category.  Consider a log twisted hyperelliptic differential as in Definition \ref{def:hyp-diff}. We further assume that $\uS$ is a geometric point. Let $G$ be the weighted graph associated to $\eta$. Then the hyperelliptic involution $\iota$ has a natural action on the underlying graph $\uG$. We have the following compatibility result. 

\begin{lemma}\label{lem:graph-inv}
The hyperelliptic involution $\iota$ acts on $G$ such that
\begin{enumerate}
 \item The partition of vertices $V(G) = V^d(G) \sqcup V^{nd}(G)$ is stable under $\iota$;
 \item The set of markings $\{\sigma_{i}\}_{i=1}^n$ is stable under $\iota$ with the contact orders satisfying $c_{\sigma_i} = c_{\iota(\sigma_i)}$; 
 \item The contact orders of the nodes satisfy $c_l = c_{\iota(l)}$ for any $l \in E(G)$.
\end{enumerate}
In particular, the action of $\iota$ on $G$ induces an isomorphism of monoids
\begin{equation}\label{equ:hyp-monoid-iso}
\phi_{\iota}: \ocM(G) \to \ocM(G)
\end{equation}
such that $\phi_{\iota}(e_l) = e_{\iota(l)}$ and $\phi_{\iota}(e_v) = e_{\iota(v)}$ for any $l \in E(G)$ and $v \in V(G)$. 
\end{lemma}
\begin{proof}
Since $\iota^* \eta = -\eta$ by Definition \ref{def:hyp-diff}, we have $\iota^{*}\bar{\eta}^{\flat} = \bar{\eta}^{\flat}$ on the level of characteristic monoids. Properties (1)--(3) then follow from the construction in Section \ref{ss:minimality}. The isomorphism $\phi_{\iota}$ follows from the construction of $\ocM(G)$ in Section \ref{ss:minimality} and the universal property of saturation.
\end{proof}

We will describe the minimal monoid in two different ways. The first one gives a description as the quotient of $\iota$-action.

\begin{lemma}\label{lem:inv-monoid-1}
Notations as above, denote by $\ocM(G,\iota)$ the coequalizer of the following diagram in the category of fine, saturated, and sharp monoids:
\[
\ocM(G) \rightrightarrows \ocM(G)
\]
where the top arrow is given by $\phi_{\iota}$ and the bottom arrow is given by the identity. Denote by $\psi: \ocM(G) \to \ocM(G,\iota)$. Then the fine, saturated, and sharp monoid $\ocM(G,\iota)$ satisfies the following properties: 
\begin{enumerate}
 \item Both $\psi(e_v)$ and $\psi(e_l)$  in $\ocM(G,\iota)$ are non-trivial for any $v \in V^{d}(G)$ and any $ l \in E(G)$;
 \item $\psi(e_v) = \psi(e_{\iota(v)})$ and $\psi(e_{l}) = \psi(e_{\iota(l)})$. 
\end{enumerate}
\end{lemma}
\begin{proof}
The second property follows from the description of the isomorphism (\ref{equ:hyp-monoid-iso}). For the first one, consider dually the equalizer of cones
\[
\ocM(G,\iota)^{\vee} \rightarrow \ocM(G)^{\vee} \rightrightarrows \ocM(G)^{\vee}.
\]
Since $\phi_{\iota}^{\vee}$ is an automorphism of the cone $\ocM(G)^{\vee}$, the image of $\ocM(G,\iota)^{\vee}$ in $\ocM(G)^{\vee}$ contains a vector in the interior of $\ocM(G)^{\vee}$. This finishes the proof. 
\end{proof}

\begin{lemma}\label{lem:min}
Notations as above, there is a canonical morphism of monoids
\[
\ocM(G,\iota) \to \ocM_{S}.
\]
\end{lemma}
\begin{proof}
Note that by definition $\iota$ acts trivially on $\cM_{S}$, hence trivially on $\ocM_{S}$. The statement thus follows from Lemma \ref{lem:inv-monoid-1} and the universal property of coequalizer.
\end{proof}

\begin{definition}\label{def:hyp-min}
A log twisted hyperelliptic differential $\eta$ over a geometric point $\uS$ is called {\em minimal} if $\ocM(G,\iota) \to \cM_S$ is an isomorphism. A family of log twisted hyperelliptic differentials is called {\em minimal} if every geometric fiber is minimal.
\end{definition}

The definition of minimality in the family case will be justified later in Proposition \ref{prop:min-open}. 

\bigskip 

We give the second description of $\ocM(G,\iota)$ which will be useful in the construction of the stack $\Hyp(\mu)$. Consider the partition
\begin{equation*}
E(G) = E(G)^{\iota}\sqcup \hat{E}(G)
\end{equation*}
where $E(G)^{\iota}$ is the set of edges fixed by $\iota$.  Denote by $\ocN^{gp}$ the torsion-free part of the quotient group 
\[
\ocM(G)^{gp} / \langle e_l = e_{\iota(l)} \ | \ l \in \hat{E}(G) \rangle. 
\]
Denote by $\ocN \subset \ocN^{gp}$ the saturated submonoid generated by the image of $\ocM(G) \to \ocN^{gp}$. 

\begin{lemma}\label{lem:inv-monoid-2}
$\ocM(G,\iota) = \ocN$. 
\end{lemma}
\begin{proof}
Consider the natural morphism
$\ocM(G) \to \ocN. $
We check that the images of $e_v$ and $e_{\iota(v)}$ under the above morphism are identical for any $v \in V(G)$. Assume that $v \neq \iota(v)$. Take a vertex $v'$ invariant under $\iota$. Let 
$l_1,\ldots, l_k$ be a sequence of edges joining $v$ with $v'$. Then $\iota(l_1), \ldots, \iota(l_k)$
is a sequence of edges joining $\iota(v)$ with $v'$. By \eqref{equ:edge-relation}, the images of $e_{v}$ and $e_{\iota(v)}$ are identical in $\ocN$. This induces a morphism 
$\ocM(G,\iota) \to \ocN.$

On the other hand, the construction of $\ocN$ induces a morphism $\ocN \to \ocM(G,\iota)$. One checks that the two morphisms 
\[
\ocM(G,\iota) \to \ocN \mbox{ and } \ocN \to \ocM(G,\iota)
\]
are inverses to each other. This finishes the proof. 
\end{proof}

\subsection{Two properties of minimality}

Consider a log twisted hyperelliptic differential $(\eta, C \to \PP_S)$ over $S$. Denote by $C_a \to \PP_a$ the associated minimal log admissible cover over a log scheme $S_a$. Denote by $C^{\sharp} \to S^{\sharp}$ the canonical log curve associated to the underlying stable curve $\uC \to \uS$ with the markings. Let $\eta_m$ over $C_m \to S_m$ be the minimal log twisted differential associated to $\eta$ over $C \to S$. We thus have a commutative diagram
\[
\xymatrix{
S \ar[r] \ar[d] & S_a \ar[d] \\
S_{m} \ar[r] & S^{\sharp}
}
\]
which induces a canonical morphism of log schemes
\begin{equation}\label{equ:to-product}
S \to S_m\times_{S^{\sharp}}S_a.
\end{equation}

\begin{lemma}\label{lem:characteristic-inv}
The family $(\eta, C \to \PP_S)$ over $S$ is minimal in the sense of Definition \ref{def:hyp-min} if and only if (\ref{equ:to-product}) is strict. 
\end{lemma}
\begin{proof}
The statement follows from Lemma~\ref{lem:inv-monoid-2} and the second description of $\ocM(G,\iota)$.
\end{proof}

\begin{proposition}\label{prop:min-open}
Given a family of log twisted hyperelliptic differentials over a log scheme $S$, the locus of minimal objects is open in $\uS$.
\end{proposition}
\begin{proof}
By Lemma \ref{lem:characteristic-inv}, the  locus of minimal objects in $S$ is the locus  where the morphism (\ref{equ:to-product}) is strict. The statement follows from the fact that being strict is an open condition. 
\end{proof}

\begin{proposition}\label{prop:min-univ}
Given a family of log twisted hyperelliptic differentials $\eta$ over  a log scheme $S$, there is (up to a unique isomorphism) a unique minimal object $\eta_m$ over a log scheme $S_m$ with a morphism $S \to S_m$, such that $\eta$ is the pull-back of $\eta_m$, and the underlying morphism $\uS \to \uS_m$ is the identity.
\end{proposition}
\begin{proof}
Note that the log twisted differential $\eta$ is the pull-back of $\eta_m$ via the following composition
\[
S \to S_m\times_{S^{\sharp}}S_a \to S_m.
\]
Denote by $T$ the log scheme with underlying $\uS$, and the log structure given by the pull-back of $S_m\times_{S^{\sharp}}S_a$. Then we obtain a log twisted differential $\eta_T$ over $T$ by pulling back $\eta_m$ via $T \to S_m$. We check that $\iota^*(\eta_T) = - \eta_T$ as $\eta_T$ can be further pulled back to $\eta$ which is hyperelliptic.  By Lemma \ref{lem:characteristic-inv}, we obtain a minimal log twisted hyperelliptic differential $\eta_T$ over $S_T$. The uniqueness of $\eta_T$ follows from the uniqueness of $\eta_m$.
\end{proof}

\subsection{Moduli of log twisted hyperelliptic differentials}\label{ss:hyp-moduli}

Recall the setting in Definition~\ref{def:log-adm}. Denote by 
\[
\cR' = \{r'_{i}\}_{i=1}^{n_1} 
\]
the set of markings fixed by the hyperelliptic involution, and 
\[
\cU' = \{(u'_j, u''_j)\}_{j=1}^{n_2}
\]
the set of pairs of markings that are interchanged by the hyperelliptic involution. Denote by 
\[
\mu_1 = (m_i)_{i=1}^{n_1} \in (2\cdot\ZZ_{\geq 0})^{n_1}
\]
where each $m_j$ is the contact order at $r_{i}$, and by
\[
\mu'_{2} = \mu''_2 = (c_j)_{j=1}^{n_2} \in (\ZZ_{>0})^{n_2}
\]
where each $c_j$ is the contact order along both $u'_j$ and $u''_j$. Note that markings in $\cR$ can have zero contact order. For simplicity, we introduce the notation 
\[
\mu = (\mu_1, \mu'_2, \mu''_2). 
\]
Denote by $\Hyp(\mu)$ the category of log twisted hyperelliptic differentials with contact orders given by $\mu$. This is a category fibered over the category of log schemes. By Proposition \ref{prop:min-univ}, it is represented by the category of minimal log twisted hyperelliptic differentials over the category of schemes with its universal minimal log structure. Note that 
\[
2g-2 = n_1 - 4.
\]
For $\Hyp(\mu)$ to be non-empty, we require that
\begin{equation}\label{equ:non-empty}
n_1 - 4 = \sum_{i=1}^{n_1} m_i + 2\cdot \sum_{j=1}^{n_2}c_j.
\end{equation}

Now consider the moduli stack of log admissible covers $\cA:=\cA(n_1, n_2)$. Denote by $\cH_{\cA}$ the Hodge bundle over $\cA$ given by the universal curve $\cC_{\cA} \to \cA$. 

\begin{proposition}\label{prop:hyp-algebracity}
The fibered category $\Hyp(\mu)$ is represented by a finite type, separated Deligne-Mumford stack with its universal minimal log structure. Furthermore, the tautological morphism to the Hodge bundle
\[
\Hyp(\mu) \to \cH_{\cA}
\]
is representable and quasi-finite. 
\end{proposition}
\begin{proof}
Let $C_{\cA}^{\sharp} \to \cA^{\sharp}$ be the canonical log curve associated to the underlying universal curve $C_{\cA} \to \cA$. Denote by $M$ the moduli of log twisted differentials over the family of curves $\uC_{\cA} \to \underline{\cA}$ with contact orders given by $\mu$. 

Consider the fiber product of log stacks
\[
M' = M \times_{\cA^{\sharp}}\cA.
\]
We notice that the morphism $M' \to M$ is representable and finite. Consider the locus $M'' \subset M'$ satisfying condition (3) of Lemma \ref{lem:graph-inv}. Since the contact order at a node remains constant in a connected family whenever the node persists, the locus $M''$ is represented by an open substack of $M'$. By a similar argument as in Lemma \ref{lem:inv-monoid-2}, for each geometric point of $M''$, the corresponding graph and the hyperelliptic involution satisfy conditions (1), (2), and (3) of Lemma \ref{lem:graph-inv}. By Lemma \ref{lem:inv-monoid-2}, the characteristic monoid $\ocM_{M'',s}$ at each geometric point $s \in M''$ is isomorphic to $\ocM(G_s,\iota)$ where $G_s$ is the graph of the fiber over $s$. Moreover, the hyperelliptic involution acts trivially on $M''$. 

Note that the locus in $M''$ satisfying the condition $\iota^*\eta = -\eta$ is represented by a closed substack of $M''$, which we denote by $M'''$. Thus, the family of log twisted differentials over $M'''$ admits the hyperelliptic involution, and is minimal in the sense of Definition \ref{def:hyp-min}. This induces a tautological morphism 
\begin{equation}\label{equ:hyp-construction1}
M''' \to \Hyp(\mu).
\end{equation}

On the other hand, the two forgetful morphisms $\Hyp(\mu) \to \cA$ and $\Hyp(\mu) \to M$ induce a morphism $\Hyp(\mu) \to M'$ which factors through the morphism 
\begin{equation}\label{equ:hyp-construction2}
\Hyp(\mu) \to M'''.
\end{equation}
We observe that the morphisms (\ref{equ:hyp-construction1}) and (\ref{equ:hyp-construction2}) are inverse to each as morphisms of fibered categories, hence it proves the representability. The separatedness and boundedness follow from Theorem~\ref{thm:log-twisted-sec}.

Finally, the above argument shows that the morphism $\Hyp(\mu) \to M$ is representable and quasi-finite, so the second claim follows from the finiteness of Theorem \ref{thm:log-twisted-sec}.
\end{proof}

\subsection{Relative properness}
\begin{proposition}\label{prop:hyp-prop}
The forgetful morphism $\Hyp(\mu) \to \cH_{\cA}$ is finite.
\end{proposition}
\begin{proof}
We use the same notations in the proof of Proposition \ref{prop:hyp-algebracity}. By the previous proposition, it suffices to show that $\Hyp(\mu) \to \cH_{\cA}$ is proper. Since the morphism $M' \to \cH_{\cA}$ is finite, it remains to show that the morphism $\Hyp(\mu) \to M'$ is proper by the valuative criterion. Namely, we want to show that there is a unique dashed arrow fitting in the following commutative diagram of solid arrows
\[
\xymatrix{
\spec K \ar[r] \ar[d] & \underline{\Hyp(\mu)} \ar[d] \\
\spec R \ar[r] \ar@{-->}[ur]& \underline{M'}. 
}
\]

Denote by $S$ the log scheme with $\uS = \spec R$, and the log structure pulled back from $M'$. Let $T$ and $t$ be the generic and closed points of $S$ with the pull-back log structures respectively. Denote by $(\eta: C \to L, C \to P)$ the family over $S$ pulled back from that over $N'$. Here $\eta$ is the log twisted differential, and $C \to P$ is the log admissible cover.  

By the proof of Proposition \ref{prop:hyp-algebracity}, it suffices to show that the closed fiber $\eta_s$ lies in the open locus $M''$. Note that contact orders remain unchanged under specialization. In particular, condition (3) of Lemma \ref{lem:graph-inv} at the conjugate nodes holds over the closed fiber. Hence $\eta_s$ lies in $M''$. This finishes the proof.
\end{proof}

\subsection{Log smoothness of moduli of log twisted hyperelliptic differentials}\label{ss:hyp-smooth}

To prove the log smoothness, we first introduce log twisted quadratic differentials over rational curves.  Consider the two sets of markings 
\[ 
\cR = \{r_{i}\}_{i=1}^{n_1}, \quad \cU = \{u_j\}_{j=1}^{n_2}.
\]
with the two sets of contact orders respectively
\[
\mu_{1} = (m_i)_{i=1}^{n_1}, \quad  \mu'_2 = (2\cdot c_j)_{j=1}^{n_2},
\]
where we require $m_i \in \NN$ and $c_j \in \NN_{>0}$. Denote by $\mu' = (\mu_{1}, \mu'_2)$ and $n = n_1 + n_2$. 

Consider the universal family of canonical log curves $\PP := \oC_{0,n} \to M := \oM_{0,n}$ with the line bundle
\begin{equation*}
\uQ = \omega_{\uPP/\uM}^{\otimes 2}\left(\sum_{i=1}^{n_1}r_i\right). 
\end{equation*}
over the underlying curve $\uPP$. We are in the situation of Section \ref{ss:log-section}. Denote by $Q$ the log stack with the underlying structure $\uQ$, where the log structure of $Q$ is obtained similarly as in~(\ref{eq:log-lift}) by combining the log structures from the zero section and from the curve $\PP$. By Theorem \ref{thm:log-twisted-sec}, we thus obtain the moduli space of log twisted quadratic differentials $\cQ := \lM(\mu')$. Denote by $\frq: \PP_{\cQ} \to Q_{\cQ}$ the universal quadratic differential over $\cQ$. For $\cQ$ to be non-empty, we require that 
\begin{equation*}
n_1 - 4 = \sum_{i=1}^{n_1}m_i + 2 \cdot \sum_{j=1}^{n_2}c_j,
\end{equation*}
which is identical to (\ref{equ:non-empty}).

\begin{proposition}\label{prop:qudratic-smooth}
The forgetful morphism $\cQ \to M$ is log smooth. In particular, the log stack $\cQ$ is log smooth.
\end{proposition}
\begin{proof}
Note that for a quadratic differential $q \in \cQ(S)$ over a log scheme $S$, the obstruction relative to $M$ is given by $H^1(q^*T_{L_\cQ/\PP_\cQ})$ where $T_{L_\cQ/\PP_\cQ}$ is the log tangent bundle of $L_{\cQ} \to \PP_\cQ$. By a direct degree calculation, we observe that $T_{L_\cQ/\PP_\cQ}$ is the trivial line bundle. Since $\PP_\cQ$ is a family of rational curves, we conclude that the obstruction $H^1(q^*T_{L_\cQ/\PP_\cQ}) = H^1(\cO)= 0$ vanishes. This proves the statement. 
\end{proof}

Consider a log twisted hyperelliptic differential $(\eta: C \to L_{C}, C \to P)$ over a log scheme $S$, where $L_{C}$ is the log scheme over $C$, with the underlying structure given by $\omega_{C/S}$, and the log structure $\cN_{L_C}$ given by the zero section and the log structure pulled back from $C$, see (\ref{eq:log-lift}).  The family $P \to S$ induces a morphism $S \to M$. Denote by $Q_{C}$ and $Q_{P}$ the log schemes over $C$ and $P$ respectively obtained by pulling back $Q$ via $S \to M$. We then have
\begin{equation*}
\uQ_{C} = \omega_{C/S}^{\otimes 2}.
\end{equation*}
The non-linear morphism
\[
\omega_{C/S} \to \omega_{C/S}^{\otimes 2}
\]
induces a morphism of log schemes
\begin{equation*}
L_{C} \to Q_{C}.
\end{equation*}
Define $\eta^{\otimes 2}$ to be the composition
\[
C \to L_{C} \to Q_{C}.
\]

\begin{lemma}\label{lem:hyp-quadratic}
The morphism $\eta^{\otimes 2}$ is invariant under $\iota$, hence descends to a log twisted quadratic differential $q : P \to L$. 
\end{lemma}
\begin{proof}
This follows from the property that $\iota^*\eta = -\eta$. 
\end{proof}

For convenience, we will denote the induced log twisted quadratic differential $q$ by $\eta^{\otimes 2}$. Thus Lemma \ref{lem:hyp-quadratic} yields a morphism of log stacks
\begin{equation}\label{equ:hyp-quadratic}
\Hyp(\mu) \to \cQ , \ \ \ \eta \mapsto \eta^{\otimes 2}. 
\end{equation}

\begin{proposition}\label{prop:hyp-smooth}
The morphism \eqref{equ:hyp-quadratic} is log \'etale. In particular, the log stack $\Hyp(\mu)$ is log smooth. 
\end{proposition}
\begin{proof}
We check the log smoothness using the local lifting property \cite[(3.3)]{KKato}. Consider the following commutative diagram of solid arrows
\begin{equation*}
\xymatrix{
T \ar[rr] \ar[d]  && \Hyp \ar[d] \ar[rr] && \cA \ar[d] \\
T' \ar[rr]  \ar@{-->}[urr] \ar@{-->}[urrrr] && \cQ \ar[rr] && M 
}
\end{equation*}
where 
\begin{enumerate}
 \item $\cA = \cA(n_1, n_2)$ is the moduli of log admissible covers; 
 \item The underlying structure $\uT' = \spec k[\epsilon]/(\epsilon^2) $ for an algebraically closed field $k$; 
 \item The morphism $T \to T'$ is a strict closed immersion with $\uT$ given by the closed point. 
\end{enumerate}
It suffices to show that there is a unique lift $T' \to \Hyp$ that makes the above diagram commutative. 

Note that $\cA \to M$ is a log \'etale morphism. Thus there is a unique lift $T' \to \cA$ making the above diagram commutative. We now arrive at a commutative diagram of solid arrows as follows
\begin{equation*}
\xymatrix{
\omega_{T} \ar[rr] \ar[d] & & \omega_{T'} \ar[d] \ar[rr] && L_{C_{T'}/T'} \ar[lld] \\
C_{T} \ar@/^/[u]^{\eta_{T}} \ar[rr] \ar[d] && C_{T'} \ar[d] \ar@{-->}@/^/[u]^{\eta_{T'}}  \ar@/_1pc/[urr]_{q}  & \\ 
T \ar[rr] && T'
}
\end{equation*}
where $\eta_{T}$ is the log twisted hyperelliptic differential given by $T \to \Hyp$, the family $C_{T'} \to T'$ is the hyperelliptic curve given by $T' \to \cA$, and $q$ is the pull-back of the quadratic differential given by $T' \to \cQ$. It suffices to show that the lift $\eta_{T'}$ of $\eta_{T}$ exists and is unique. 

We next construct such a lift. First observe that the existence and uniqueness of $\eta_{T'}$ are obvious over the non-degenerate components of the curve. On the degenerate locus, the underlying morphism $\eta_{T'}$ has to be a constantly zero section, hence is also unique. It remains to consider the part of log structures. 

Denote by $G$ the graph associated to $\eta_T$. Fixing a chart $\ocM_T \to \cM_{T'}$, we may identify elements of $\ocM(G)$ with their images in $\cM_{T'}$. Note that $T$ and $T'$ have the same topological space. For any point $x \in C_{T}$, take a small connected neighborhood $U$ of $x$ in $T'$ such that either $x$ is a smooth unmarked point of $C_{T'}$ and $U$ contains no special point, or $x$ is the only special point contained in $U$. By the discussion in Section~\ref{ss:local-structure}, there is a vertex $v$ corresponding to a component containing $x$ such that over $U$, the morphism
\[
\eta^{\flat}_{T} - e_v : \cM_{\omega_{T}}|_{U} \to \cM_{C_{T}}|_{U}, \ \ \ \delta \mapsto \eta^{\flat}_{T}(\delta) - e_v 
\]
induces a non-degenerate meromorphic section of the underlying bundle of $\omega_{T}$. Hence the morphism
\[
q^{\flat}_{T} - 2\cdot e_v : \cM_{L_{C_{T'}/T'}}|_{U} \to \cM_{C_{T}}|_{U}, \ \ \ \delta \mapsto \eta^{\flat}_{T}(\delta) - 2\cdot e_v 
\]
also induces a non-degenerate meromorphic section of the underlying bundle of $L_{C_{T'}/T'}$. Consequently there is a unique square root of the meromorphic section induced by $q^{\flat}_{T} - 2\cdot e_v $, which we denote by $\ueta_{T'}|_{U}$, such that $\iota^*\ueta_{T'}|_{U} = - \ueta_{T'}|_{U}$. Define $\eta^{\flat}_{T'}|_{U} = \ueta_{T'}|_{U} + e_v$. Since the choice of $\eta^{\flat}_{T'}|_{U}$ exists and is unique locally, it can be glued together and form a global morphism $\eta^{\flat}_{T'} : \cM_{\omega_{T}} \to \cM_{C_{T}}$.  This defines the log twisted hyperelliptic differential $\eta_{T'}$. The compatibility with $\eta_{T}$ follows from the construction. 

Finally, we remark that since the involution $\iota$ fixes the base log structure and since $\eta^{\otimes 2}_{T'} = q$, the outcome $\eta_{T'}$ is independent of the choice of charts. This finishes the proof.
\end{proof}


\newcommand{\etalchar}[1]{$^{#1}$}
\providecommand{\bysame}{\leavevmode\hbox to3em{\hrulefill}\thinspace}
\providecommand{\MR}{\relax\ifhmode\unskip\space\fi MR }
\providecommand{\MRhref}[2]{%
  \href{http://www.ams.org/mathscinet-getitem?mr=#1}{#2}
}
\providecommand{\href}[2]{#2}

\end{document}